\documentclass{article}
\usepackage{amssymb}
\usepackage{amsmath}

\newtheorem{theorem}{Theorem}

\newtheorem{condition}[theorem]{Condition}

\newtheorem{corollary}[theorem]{Corollary}

\newtheorem{definition}[theorem]{Definition}

\newtheorem{lemma}[theorem]{Lemma}

\newtheorem{proposition}[theorem]{Proposition}
\newtheorem{remark}[theorem]{Remark}

\newenvironment{proof}[1][Proof]{\noindent \textbf{#1.} }{\  \rule{0.5em}{0.5em}}
\begin{document}

\title{Existence and stability results on a class of Non Linear Schr\"{o}%
dinger Equations in bounded domains with Dirichlet boundary conditions}
\author{Marco Ghimenti$^{1}$, Dimitrios Kandilakis$^{2}$, Manolis
Magiropoulos$^{3}$ \\
\\
$^{1}$Dipartimento di Matematica,\\
Universit\`{a} di Pisa,\\
Largo Buonarroti 1/c,\\
56127 Pisa, Italy \\
ghimenti@mail.dm.unipi.it\\
\newline
\\
$^{2}$School of Architectural Engineering\\
Techical University of Crete,\\
73100 Chania, Greece\\
dimkand@gmail.com\qquad \\
\newline
\\
$^{3}$Technological Educational Institute of Crete\\
Department of Electrical Engineering\\
71500 Heraklion, Crete, Greece\\
mageir@staff.teicrete.gr}
\date{}
\maketitle

\begin{abstract}
Existence of solution and $L^{2}$and $H^{1}$localization results on a class
of Non Linear Schr\"{o}dinger type equations with a bounded nonlinearity are
obtained, for a bounded domain and with Dirichlet boundary conditions. The
kind of stability under discussion shows that the corresponding solution
exhibits features of a solitary wave type.



The first author is partial supported by G.N.A.M.P.A. 
\end{abstract}

\section{Introduction}

We study the existence, stability and localization of soliton type solutions
for the Non Linear Schr\"{o}dinger Equation (briefly, NLSE) in the
semiclassical limit (that is for $h\rightarrow 0^{+}$), for a bounded domain
with Dirichlet boundary conditions.

In our framework, the problem takes the form%
\begin{equation*}
ih\frac{\partial \psi }{\partial t}=-\frac{h^{2}}{2}\Delta \psi +\frac{1}{%
2h^{\alpha }}W^{\prime }(\left \vert \psi \right \vert )\frac{\psi }{\left
\vert \psi \right \vert }+V(x)\psi \text{ , }\psi \in C^{1}(%
\mathbb{R}
_{0}^{+},H_{0}^{1}(\Omega ,%
\mathbb{C}
))
\end{equation*}%
\begin{equation}
\psi (0,x)=\phi _{h}(x)\text{, }x\in \Omega \text{,}  \label{h}
\end{equation}

\begin{equation*}
\psi (t,x)=0\text{ on }%
\mathbb{R}
_{0}^{+}\times \partial \Omega \text{, }
\end{equation*}%
$\Omega \subset 
\mathbb{R}
^{N}$ being open and bounded, $N\geqq 3$, $\alpha >0$, where $\phi
_{h}(x)\in $ $H_{0}^{1}(\Omega )$ is a suitable initial datum, and $V$ is an
external potential. Conditions for the nonlinear term $W$ and the potential $%
V$ are to be precised and discussed in the following sections.

The NLSE in the presence of a potential is largely present in literature. In
particular, it has been extensively studied the effect of the potential $V$
on the existence and the profile of a stationary solution, that is a
solution of the form $\psi (t,x)=U(x)e^{-\frac{i}{h}\omega t}$, $\omega
=\lambda /2$, where $U$ solves the equation%
\begin{equation*}
-h^{2}\Delta U+\frac{1}{h^{\alpha }}W^{\prime }(U)+2V(x)U=\lambda U\text{.}
\end{equation*}%
The first attempt to this direction is the work of Floer and Weinstein \cite%
{F-W}, for the one dimensional cubic NLSE (with a generalization for higher
dimensions and different nonlinearity in \cite{O}) where, by means of a
Lyapunov - Schmidt reduction, it is proved that, if $V$ has a non degenerate
minimum, then a stationary solution exists, and this solution has a peak
located at this minimum. Del, Pino and Felmer \cite{DP-F} showed that any
(possibly degenerate) minimum of $V$ generates a stationary solution. We
also mention \cite{A-B-C, Li}, in which similar results are obtained with
different techniques.Concerning global methods, in \cite{Ra} Rabinowitz
proved the existence of a stationary solution with a Mountain Pass argument.
Later, Cingolani and Lazzo \cite{Ci-La} proved that the Lusternik -
Schnirelmann category of the minimal level of $V$ gives a lower bound for
the number of stationary solutions. The topological approach was also
adopted in \cite{Am-Ma-Se}, where a more refined topological invariant is
used, and in \cite{Ben-Gri-Mich}, where the presence of a negative potential
allows the existence of a solution in the so called ``zero mass" case.

Another interesting feature is the influence of the domain in the stationary
NLSE, when $V=0$. In this case, a single-peaked solution can be constructed.
In \cite{Ni-Wei}, Ni and Wei showed that the least energy solution for the
equation 
\begin{equation}
-h^{2}\Delta u+u=u^{p}\text{, }u>0\text{ in }\Omega \text{,}  \label{1}
\end{equation}%
with homogeneous Dirichlet boundary condition, has a unique peak, located at
a point $P_{h}$ with $d(P_{h},\partial \Omega )\underset{h\rightarrow 0^{+}}{%
\rightarrow }\max_{\Omega }d(P,\partial \Omega )$. Later, Wei \cite{Wei}
proved a result that can be viewed as the converse of the forementioned
theorem. Namely, the author showed that for any local maximum $P$ of the
distance from the boundary $\partial \Omega $, one can construct a
single-peaked solution of (\ref{1}) whose peak tends to $P$ as $h\rightarrow
0^{+}$. The profile of the solution is, up to rescaling, close to the
profile of the ground state solution of the limit problem 
\begin{equation*}
-\Delta u+u=u^{p}\text{, }u>0\text{, }
\end{equation*}%
in the whole of $%
\mathbb{R}
^{N}$. We also mention \cite{Ca-Da-Nou-Yan} in which the existence of a
multi-peaked solution of (\ref{1}) is proved.

In the present work we follow a different approach, incorporating and
exploiting ideas found in \cite{Bel-B-G-M, B-G-M, Ben-Ghi-Mich}, where the
problem (\ref{h}) had been studied for the whole of $%
\mathbb{R}
^{N}$ for both cases: $V=0$ (existence and stability), and $V\neq 0$
(existence, stability and dynamics).

We want to point out two main differences between this paper and \cite%
{Bel-B-G-M}. First, dealing with a bounded domain and a bounded nonlinearity
gives us enough compactness to easily prove an orbital stability result. In
particular our result is also true for positive nonlinearities, which is
forbidden in the whole space by Pohozaev's Theorem.

This difference is less evident in what comes after, since, when dealing
with the semiclassical limit, we have to face a limit problem in the whole
space $\mathbb{R}^{N}$, which is the same of \cite{Bel-B-G-M}, so we have to
reintroduce some hypothesis.

However, our orbital stability result could be used when studying other
situations, e.g. fixing $h$ and looking for solitons of prescribed, possibly
large, $L^{2}$ norm (the so called large solitons). If one is interested in
these topics, we recommend the nice paper by Noris, Tavares, Verzini \cite%
{Nor-Tav-Ver}, which deals with orbital stability for solitons with
prescribed $L^{2}$ norms in bounded domains, and the references therein. We
point out that in \cite{Nor-Tav-Ver}, the authors work with pure power
nonlinearities, so they can have some compactness loss -for $L^{2}$-critical
and supercritical powers- even if the domain is bounded. In particular, for
critical and supercritical powers they have orbital stability only for small 
$L^{2}$ norms.

The second point in which the bounded domain marks a difference with the
paper \cite{Bel-B-G-M}, is pointed out in the Appendix. In fact, when trying
to describe dynamics, immediately it appears a repulsive effect of the
boundary. The soliton is affected by a force oriented with the inward normal
to the boundary. Unfortunately, the value of the force depends on the $%
D^{1,2}$ norm of the solution on the boundary, so, at the moment, we were
not able to give a quantitative estimate of this repulsive force, and
further efforts are needed.

According to this line of thought, we have divided the present work into
three sections and an appendix:

In Section 2, existence and orbital stability results are obtained for the
case $V=0$, by referring to the related eigenvalue problem

\begin{eqnarray}
-\Delta U+W^{\prime }(U) &=&\lambda U\text{, in }\Omega  \notag \\
&&  \label{s} \\
U &\equiv &0\text{, on }\partial \Omega \text{,}  \notag
\end{eqnarray}%
given that a solution $U(x)$ in $H_{0}^{1}(\Omega )$ of (\ref{s}) results to
a solution $\psi =U(x)e^{-\frac{i\lambda t}{2}}$ of (\ref{h}), with initial
condition $\psi (0,x)=U(x)$. These results are summarized in Proposition 6.
One should notice that the relative proofs work without having to impose the
usual restriction $2<p<2+\frac{4}{N}$, by relaxing the restriction to $%
2<p<2^{\ast }=\frac{2N}{N-2}$ instead. It is the boundedness of the domain
that allows us in doing so.

In Section 3, where we assume the presence of an external potential, our
basic result, obtained by means of a rescaling procedure, is to prove $L^{2}$
localization in the sense that if we start with an initial datum close to a
ground state solution $U$ of 
\begin{eqnarray}
-h\Delta U+\frac{1}{h^{\alpha +1}}W^{\prime }(U) &=&\frac{\omega }{h^{\alpha
+1}}U\text{, in }\Omega  \notag \\
&& \\
U &\equiv &0\text{, on }\partial \Omega \text{,}  \notag
\end{eqnarray}%
the corresponding solution of (\ref{h}) will keep its $L^{2}$ profile along
the motion, provided that $h$ is sufficiently small. Here and in what
follows, the restriction $2<p<2+\frac{4}{N}$ is imposed, since we need to
face a limit problem in $%
\mathbb{R}
^{N}$.

In Section 4, an $H^{1}$ modular localization result is obtained for the
case $V\neq 0$, and for both cases: the unbounded and the bounded one. When
we work on the whole of $%
\mathbb{R}
^{N}$, we start with a ground state solution $U_{1}$ of the $%
\mathbb{R}
^{N}$ counterpart of (\ref{s}), proving that a solution of the $%
\mathbb{R}
^{N}$ counterpart of (\ref{h}), with initial condition close to $U_{1}$,
preserves its basic modular $H^{1}$ profile as time passes, in the sense
that given $\varepsilon >0$, for all $t\geq 0$, the ratio of the squared $%
L^{2}$ norm of $\left \vert \nabla u_{h}(t,x)\right \vert $ with respect to
the complement of a suitable open ball over the squared $L^{2}$ norm of $%
\left \vert \nabla u_{h}(t,x)\right \vert $ with respect to the whole of $%
\mathbb{R}
^{N}$ is less than $\varepsilon $ for $h$ sufficiently small, where $%
u_{h}(t,x)$ is taken by the polar expression of $\psi (t,x)$, namely $\psi
(t,x)=$ $u_{h}(t,x)e^{is_{h}(t,x)}$. The bounded case is treated by
exploiting ideas developed for the $L^{2}$ problem (Section 3).

Finally, and as we mentioned before, in the Appedix it is described an
attempt to study dynamics in the frame of a bounded domain, where we
encountered difficulties due to computational complications related to the
action of $\nabla V$ on the motion as well as to the repulsive effect of the
boundary.

\section{ The case $V=0$}

\subsection{Existence}

For simplicity of the exposition we assume $h=1$. As it has been already
said, the case $V=0$ is related to problem (\ref{s}), and a solution $u(x)$
in $H_{0}^{1}(\Omega )$ of (\ref{s}) results to a solution $\psi =u(x)e^{-%
\frac{i\lambda t}{2}}$ of (\ref{h}) with initial condition $\psi (0,x)=u(x)$.

Notice that a minimizer of 
\begin{equation*}
J(u)=\int \nolimits_{\Omega }\left( \frac{1}{2}\left \vert \nabla u\right
\vert ^{2}+W(u)\right) dx
\end{equation*}%
on $S_{\sigma }=\left \{ u\in H_{0}^{1}(\Omega ):\left \Vert u\right \Vert
_{L^{2}(\Omega )}=\sigma \right \} $, for some fixed $\sigma >0$, is a
solution of (\ref{s}), for suitable $\lambda $. Thus we focus on the
existence of such a minimizer. We impose on $W$ the following conditions:

\begin{condition}
$W$ is a $C^{1}$, bounded and even map $%
\mathbb{R}
\rightarrow 
\mathbb{R}
$.
\end{condition}

\begin{condition}
\bigskip $\left \vert W^{\prime }(s)\right \vert \leq c|s|^{p-1}$, $%
2<p<2^{\ast }=\frac{2N}{N-2}$, where $c$ is a suitable positive constant.
\end{condition}

\begin{remark}
\textit{Although it is intuitively quite clear the construction of such
maps, an easy concrete example is furnished by choosing }$W=\sin (s^{p})$%
\textit{, for }$p\geq 0$\textit{, and evenly expanding it on the whole of }$%
\mathbb{R}
$\textit{. We also stress the fact that a bounded }$W$\textit{\ is related
to the global well posedness results by Cazenave.}
\end{remark}

Notice that 
\begin{equation}
-\infty <\mu =\underset{u\in S_{\sigma }}{\inf }J(u)\text{.}  \label{s1}
\end{equation}%
If $\left \{ u_{n}\right \} $ is a minimizing sequence in $S_{\sigma }$ for $%
J(u)$, that is $J(u_{n})\rightarrow \mu $, it is evident that $\left \{
u_{n}\right \} $ is bounded in $H_{0}^{1}(\Omega )$, thus, up to a
subsequence, $u_{n}\rightharpoonup \overline{u}\in H_{0}^{1}(\Omega )$, and $%
u_{n}\rightarrow \overline{u}$ in $L^{2}(\Omega )$. The latter implies $%
\left \Vert u_{n}\right \Vert _{L^{2}(\Omega )}\rightarrow \left \Vert 
\overline{u}\right \Vert _{L^{2}(\Omega )}$, thus $\overline{u}\in S_{\sigma
} $.

Next, we obtain a similar result to Proposition 11 in \cite{Bel-B-G-M},

\begin{proposition}
If $\left \{ w_{n}\right \} $ is a minimizing sequence in $S_{\sigma }$ for $%
J$, that is $J(w_{n})\rightarrow \mu $, satisfying the constrained P - S
condition, that is, there exists a real sequence $\lambda _{n}$ of Lagrange
multipliers such that%
\begin{equation}
-\Delta w_{n}+W^{\prime }(w_{n})-\lambda _{n}w_{n}=\sigma _{n}\rightarrow 0%
\text{,}  \label{s2}
\end{equation}%
then $\lambda _{n}$ is bounded.
\end{proposition}

\begin{proof}
\bigskip Since, as we saw above, $w_{n}$ is bounded in $H_{0}^{1}(\Omega )$,
(\ref{s2}) implies%
\begin{equation*}
\left \vert \int \nolimits_{\Omega }\left( \left \vert \nabla w_{n}\right
\vert ^{2}+W^{\prime }(w_{n})w_{n}-\lambda _{n}w_{n}^{2}\right) dx\right
\vert \leq \left \Vert \sigma _{n}\right \Vert _{\ast }\left \Vert
w_{n}\right \Vert _{H_{0}^{1}(\Omega )}\rightarrow 0\text{,}
\end{equation*}%
where by $\left \Vert \cdot \right \Vert _{\ast }$ is denoted the dual norm
for $H_{0}^{1}(\Omega )$. We have 
\begin{eqnarray*}
\int \nolimits_{\Omega }\left( \left \vert \nabla w_{n}\right \vert
^{2}+W^{\prime }(w_{n})w_{n}-\lambda _{n}w_{n}^{2}\right) dx &=& \\
\int \nolimits_{\Omega }\left( \left \vert \nabla w_{n}\right \vert
^{2}+2W(w_{n})-2W(w_{n})+W^{\prime }(w_{n})w_{n}-\lambda
_{n}w_{n}^{2}\right) dx &=& \\
2J(w_{n})-\lambda _{n}\sigma ^{2}+\int \nolimits_{\Omega }\left( W^{\prime
}(w_{n})w_{n}-2W(w_{n})\right) dx &\rightarrow &0.
\end{eqnarray*}%
Notice that $J(w_{n})$ is bounded, and because of Condition 2, 
\begin{eqnarray*}
\left \vert \int \nolimits_{\Omega }\left( W^{\prime
}(w_{n})w_{n}-2W(w_{n})\right) dx\right \vert &\leq & \\
\int \nolimits_{\Omega }\left \vert W^{\prime }(w_{n})w_{n}\right \vert
dx+2\int \nolimits_{\Omega }\left \vert W(w_{n})\right \vert dx &\leq & \\
c_{1}\left \Vert w_{n}\right \Vert _{H_{0}^{1}(\Omega )}^{p}+2kmeas\left(
\Omega \right) &<&+\infty \text{,}
\end{eqnarray*}%
where $k$ is an upper bound of $|W|$. Thus $\lambda _{n}$ is bounded.
\end{proof}

By Ekeland's principle, if $\left \{ u_{n}\right \} $ is a minimizing
sequence in $S_{\sigma }$ for $J(u)$, we may assume that it satisfies the\
constrained P - S condition, that is, there exists a real sequence $\lambda
_{n}$ so that (\ref{s2}) holds. Because of Proposition 4, $\lambda _{n}$ is
bounded, and the following hold%
\begin{eqnarray*}
\lambda _{n} &\rightarrow &\lambda \\
u_{n} &\rightharpoonup &u\text{ in }H_{0}^{1}(\Omega ) \\
u_{n} &\rightarrow &u\text{ in }L^{p}(\Omega )\text{ for }1\leq p<2^{\ast }%
\text{.}
\end{eqnarray*}%
We have already shown that $u\in S_{\sigma }$. Thus, $u\neq 0$. Next, we
show that%
\begin{equation}
-\Delta u+W^{\prime }(u)=\lambda u\text{.}  \label{s3}
\end{equation}%
To this end, if $\varphi $ is a test function, combining the three
considerations above with Condition 2, we have 
\begin{eqnarray*}
\int \nolimits_{\Omega }\nabla u_{n}\nabla \varphi dx &\rightarrow &\int
\nolimits_{\Omega }\nabla u\nabla \varphi dx \\
\int \nolimits_{\Omega }W^{\prime }(u_{n})\varphi dx &\rightarrow &\int
\nolimits_{\Omega }W^{\prime }(u)\varphi dx \\
\lambda _{n}\int \nolimits_{\Omega }u_{n}\varphi dx &\rightarrow &\lambda
\int \nolimits_{\Omega }u\varphi dx\text{,}
\end{eqnarray*}%
implying (\ref{s3}).

\bigskip Notice next that due to Condition 2, the Nemytskii operator%
\begin{equation*}
W:L^{t}\left( \Omega \right) \rightarrow L^{1}\left( \Omega \right) \text{, }%
2<t<2^{\ast }\text{,}
\end{equation*}%
is continuous, whereas $u_{n}\rightarrow u$ in $L^{t}(\Omega )$, for $%
2<t<2^{\ast }$. Thus,%
\begin{equation*}
\mu \leq J(u)=\frac{1}{2}\int \nolimits_{\Omega }\left \vert \nabla u\right
\vert ^{2}dx+\int \nolimits_{\Omega }W(u)dx\leq \underset{n\rightarrow
\infty }{\lim }J(u_{n})=\mu \text{,}
\end{equation*}%
proving that $J(u)=\mu $.

This completes the proof for the existence of a non trivial solution of (\ref%
{s}), for suitable $\lambda $. In fact, the weak convergence $%
u_{n}\rightharpoonup u$ turns out to be a strong one: Since $J(u_{n})-J(u)$, 
$\int \nolimits_{\Omega }(W(u_{n})-W(u))dx\rightarrow 0$, we obtain $\left
\Vert u_{n}\right \Vert _{H_{0}^{1}(\Omega )}\underset{}{\rightarrow }\left
\Vert u\right \Vert _{H_{0}^{1}(\Omega )}$, thus proving that $%
u_{n}\rightarrow u$ in $H_{0}^{1}(\Omega )$. Since $W$ has been assumed
even, we may take a non trivial nonnegative solution of (\ref{s}). By
Harnack's inequality, this solution is strictly positive on $\Omega $. We
thus obtain a positive solution $\overline{u}\in S_{\sigma }$ for problem (%
\ref{s}), for suitable $\lambda $. The wave function 
\begin{equation}
\psi (t,x)=\overline{u}(x)e^{^{-i\omega t}}\text{, }\omega =\lambda /2
\end{equation}%
is a stationary solution of (\ref{h}), for $h=1$, $V\equiv 0$, with initial
condition $\phi (x)=\psi (0,x)=\overline{u}(x)$. Evidently, $-\overline{u}%
(x)e^{-i\omega t}$, $\omega =\lambda /2$, is a stationary solution of (\ref%
{h}), too.

\subsection{Stability}

We turn next our attention to the stability of the stationary solution. To
this end, we focus on the reduced form of (\ref{h}), 
\begin{equation*}
2i\frac{\partial \psi }{\partial t}=-\Delta \psi +W^{\prime }(\left \vert
\psi \right \vert )\frac{\psi }{\left \vert \psi \right \vert }\text{ in }%
\mathbb{R}
_{0}^{+}\times \Omega \text{,}
\end{equation*}%
\begin{equation}
\psi (0,x)=\phi (x)\text{,}  \label{15}
\end{equation}%
\begin{equation*}
\psi (t,x)=0\text{ on }%
\mathbb{R}
_{0}^{+}\times \partial \Omega \text{, }
\end{equation*}%
by taking, as it was mentioned above, $h=1$. The different time slices $\psi
_{t}(x)$ of each solution of (\ref{15}), where such a solution may be
understood as the time evolution of some initial condition $\psi _{0}(x)$,
could be thought of as elements of a proper phase space $X\subset
L^{2}(\Omega ,%
\mathbb{C}
)$, with the set 
\begin{equation}
\Gamma =\left \{ u(x)e^{i\theta }\text{, }\theta \in 
\mathbb{R}
/2\pi 
\mathbb{Z}
\text{, }u\in S_{\sigma }\text{, }J(u)=\mu =\underset{w\in S_{\sigma }}{\inf 
}J(w)\right \}
\end{equation}%
being an invariant (under evolution) manifold of $X$. Evidently, $\pm 
\overline{u}(x)\in \Gamma $.

To make the description of all this more clear, one should notice that if $%
\psi _{t_{0}}(x)$ is a time slice of a solution $\psi (t,x)$ of (\ref{15}),
the evolution map is defined by%
\begin{equation*}
U_{t}\psi _{t_{0}}(x)=\psi _{t_{0}+t}(x)\text{,}
\end{equation*}%
meaning that this time slice might be considered as the initial condition of
the solution $\psi _{1}(t,x)=\psi (t+t_{0},x)$. Now, if $u(x)e^{i\theta }\in
\Gamma $, then $u$ is a solution of (\ref{s}), with suitable $\lambda $,
and, at the same time, $u(x)e^{i\theta }$ is the initial condition of the
solution $\psi (t,x)=u(x)e^{i(\theta -\lambda t/2)}$ of (\ref{15}). Since $%
u(x)e^{i(\theta -\lambda t/2)}\in \Gamma $ for each $t\geqslant 0$, the
invariance of $\Gamma $ follows. We are going to prove \textit{orbital}
stability of $\overline{u}(x)e^{^{-i\omega t}}$, $\omega =\lambda /2$,
following the definition of orbital stability found in \cite{Caz-Lions},
meaning that $\Gamma $ is stable in the following sense:%
\begin{equation*}
\forall \varepsilon >0\text{, }\exists \text{ }\delta >0\text{ such that if }%
\psi (t,x)\text{ is a solution of (\ref{15}) satisfying}
\end{equation*}

$\  \  \  \  \  \  \  \  \  \  \  \  \  \  \  \  \  \ $ $\  \  \  \  \  \  \  \  \  \  \  \  \  \  \  \  \  \
\  \  \  \  \  \  \  \  \ $%
\begin{equation}
\underset{\widetilde{w}\in \Gamma }{\inf }\left \Vert \psi (0,x)-\widetilde{w%
}\right \Vert _{H_{0}^{1}(\Omega )}<\delta \text{, then }\forall t\geqq 0%
\underset{\widetilde{w}\in \Gamma }{\inf }\left \Vert \psi (t,x)-\widetilde{w%
}\right \Vert _{H_{0}^{1}(\Omega )}<\varepsilon \text{.}  \label{r}
\end{equation}

The name \textit{orbital }could be misleading in the present framework: our
nonlinearity does not ensure the uniqueness of the ground state. Moreover,
the problem is not invariant under translations. Thus it could be that the
set $\Gamma $ reduces to a single function or to \ a discrete set of ground
states (up to multiplication by a unitary complex number). Anyhow, we follow
the approach by \cite{Caz-Lions} and by \cite{Bel-B-G-M}, so we will keep
the term orbital stability. In Lemmas 21 and 22, we will see how the orbital
stability result leads to a precise description of the modulus of a solution 
$\psi (t,x)$ starting from a suitable initial datum.

Notice that $\Gamma $ is bounded in $H_{0}^{1}(\Omega )$, since for each of
its elements $\widetilde{w}=w(x)e^{i\theta }$, $w(x)$ is a constrained
minimizer of $J$, whereas $W$ is bounded. Notice that we may take $w(x)>0$.

Suppose $\Gamma $ is not stable. Then $\exists $ $\varepsilon >0$, and
sequences $\delta _{n}\rightarrow 0^{+}$, $\psi _{n}(t,x)$ of solutions of (%
\ref{15}), and $t_{n}\geq 0$ such that 
\begin{equation}
\underset{\widetilde{w}\in \Gamma }{\inf }\left \Vert \psi _{n}(0,x)-%
\widetilde{w}\right \Vert _{H_{0}^{1}(\Omega )}<\delta _{n}\text{, }\underset%
{\widetilde{w}\in \Gamma }{\inf }\left \Vert \psi _{n}(t_{n},x)-\widetilde{w}%
\right \Vert _{H_{0}^{1}(\Omega )}\geq \varepsilon \text{.}  \label{19}
\end{equation}%
Notice that the first inequality of (\ref{19}) implies%
\begin{equation*}
\underset{\widetilde{w}\in \Gamma }{\inf }\left \Vert \psi _{n}(0,x)-%
\widetilde{w}\right \Vert _{_{L^{2}(\Omega )}}<C\delta _{n}\rightarrow 0%
\text{,}
\end{equation*}%
where $C$ is the Sobolev constant satisfying $\left \Vert \cdot \right \Vert
_{L^{2}(\Omega )}\leq C\left \Vert \cdot \right \Vert _{H_{0}^{1}(\Omega )}$%
. Thus, we may obtain a sequence $\widetilde{w}_{n}$ in $\Gamma $, such that%
\begin{equation}
\left \Vert \psi _{n}(0,x)-\widetilde{w}_{n}\right \Vert _{_{L^{2}(\Omega
)}}\rightarrow 0\text{.}  \label{20}
\end{equation}%
We express now $\psi _{n}(t,x)$ in polar form, namely $\psi
_{n}(t,x)=u_{n}(t,x)e^{is_{n}(t,x)}$, with $u_{n}(t,x)=$ $|\psi _{n}(t,x)|$, 
$\forall t\geqq 0$. Since $\left \Vert \widetilde{w}_{n}\right \Vert
_{_{L^{2}(\Omega )}}=\sigma $, (\ref{20}) implies that $u_{n}(0,x)$ is
bounded in $L^{2}(\Omega )$, and at least up to a subsequence, still denoted
by $u_{n}(0,x)$, $\left \Vert u_{n}(0,x)\right \Vert _{_{L^{2}(\Omega
)}}\rightarrow M\geq 0$. Rewriting (\ref{20}) in its squared form, and
taking into consideration that 
\begin{equation*}
\int \nolimits_{\Omega }u_{n}(0,x)|w_{n}(x)|dx\leq \left \Vert
u_{n}(0,x)\right \Vert _{_{L^{2}(\Omega )}}\sigma \text{,}
\end{equation*}%
we take%
\begin{equation*}
0\geq M^{2}-2M\sigma +\sigma ^{2}=(M-\sigma )^{2}\text{,}
\end{equation*}%
thus obtaining $M=\sigma $.

The polar form $\psi (t,x)=u(t,x)e^{is(t,x)}$, turns (\ref{15}) into the
system%
\begin{equation*}
-\frac{\Delta u}{2}+\frac{W^{\prime }(u)}{2}+\left( \frac{\partial s}{%
\partial t}+\frac{1}{2}\left \vert \nabla s\right \vert ^{2}\right) u=0
\end{equation*}%
\begin{equation*}
\partial _{t}u^{2}+\nabla \cdot (u^{2}\nabla s)=0\text{, in }%
\mathbb{R}
_{0}^{+}\times \Omega \text{,}
\end{equation*}%
\begin{equation}  \label{21}
\end{equation}%
\begin{equation*}
u(0,x)e^{is(0,x)}=\phi (x)\text{, }u(t,x)=0\text{ on }%
\mathbb{R}
_{0}^{+}\times \partial \Omega \text{, }
\end{equation*}%
with the two equations of (\ref{21}) being the Euler - Lagrange equations of
the action functional%
\begin{equation}
A(u,s)=\frac{1}{4}\iint \left \vert \nabla u\right \vert ^{2}dxdt+\frac{1}{2%
}\iint W(u)dxdt+\frac{1}{2}\iint \left( \frac{\partial s}{\partial t}+%
\frac{1}{2}\left \vert \nabla s\right \vert ^{2}\right) u^{2}dxdt\text{.}
\label{22}
\end{equation}%
The total energy is given by%
\begin{equation}
E(\psi )=E(u,s)=\int \nolimits_{\Omega }\left( \frac{1}{2}\left \vert
\nabla u\right \vert ^{2}+\frac{1}{2}u^{2}\left \vert \nabla s\right \vert
^{2}+W(u)\right) dx\text{,}  \label{23}
\end{equation}%
that is,%
\begin{equation}
E(\psi )=E(u,s)=J(u)+\frac{1}{2}\int \nolimits_{\Omega }u^{2}\left \vert
\nabla s\right \vert ^{2}dx  \label{a}
\end{equation}%
Independence of time for the energy and for the charge imply that for a
solution $\psi (t,x)=u(t,x)e^{is(t,x)}$ of (\ref{15}), it holds%
\begin{equation}
\frac{d}{dt}\int \nolimits_{\Omega }u(t,x)^{2}dx=0  \label{24}
\end{equation}%
\begin{equation}
\frac{d}{dt}E(u,s)=0\text{.}  \label{25}
\end{equation}%
Equivalently, (\ref{24}) and (\ref{25}) can be expressed as%
\begin{equation}
\left \Vert \psi (t,x)\right \Vert _{_{L^{2}(\Omega )}}=\left \Vert \phi
(x)\right \Vert _{_{L^{2}(\Omega )}}  \label{26}
\end{equation}%
\begin{equation}
E(\psi (t,x))=E(\phi (x))  \label{27}
\end{equation}%
for all $t\geq 0$. Noteworthy, for stationary solution, (\ref{a}) yields $%
E(\psi )=J(u)$.

Returning to the sequence $\psi _{n}(t,x)$ satisfying (\ref{19}), we may
assume, as we saw, that $\left \Vert u_{n}(0,x)\right \Vert _{_{L^{2}(\Omega
)}}\rightarrow \sigma $, that is $\left \Vert u_{n}(t,x)\right \Vert
_{_{L^{2}(\Omega )}}\rightarrow \sigma $, for $t\geq 0$, because of (\ref{26}%
). We want to show that $\left \{ u_{n}\right \} _{n}$ is a minimizing
sequence for the functional $J$ on the constraint $\left \Vert u\right \Vert
_{_{L^{2}(\Omega )}}=\sigma $.

One should notice that the first inequality of (\ref{19}), combined with the
boundedness of $\Gamma $ ensure that $\psi _{n}(0,x)$ is bounded in $%
H_{0}^{1}(\Omega )$. Since $W$ is bounded, (\ref{a}) ensures that $E(\psi
_{n}(0,x))$ is bounded, and because of (\ref{27}), $E(\psi _{n}(t,x))$ is
bounded, for all $n$ and all $t\geq 0$. In particular, $E(\psi
_{n}(t_{n},x)) $ is bounded. A new application of (\ref{a}), ensures now
that $u_{n}(t_{n},x)$ is bounded in $H_{0}^{1}(\Omega )$. The sequence $%
\widehat{u}_{n}(t_{n},x)=\alpha _{n}u_{n}(t_{n},x)$, where $\alpha _{n}=%
\frac{\sigma }{\left \Vert u_{n}(t_{n},x)\right \Vert _{_{L^{2}(\Omega )}}}$%
, is in $S_{\sigma }$. We have, writing for simplicity $u_{n}$, $\widehat{u}%
_{n}$ instead of $u_{n}(t_{n},x)$, $\widehat{u}_{n}(t_{n},x)$, respectively,
for suitable $l_{n}=l_{n}(x)\in \left( 0\text{, }1\right) $, and because of
Condition 2, 
\begin{eqnarray*}
\left \vert J(\widehat{u}_{n})-J(u_{n})\right \vert &\leq &\frac{1}{2}%
|\alpha _{n}^{2}-1|\int \nolimits_{\Omega }\left \vert \nabla u_{n}\right
\vert ^{2}dx+\int \nolimits_{\Omega }\left \vert W(\widehat{u}%
_{n})-W(u_{n})\right \vert dx \\
&=&\frac{1}{2}|\alpha _{n}^{2}-1|\int \nolimits_{\Omega }\left \vert \nabla
u_{n}\right \vert ^{2}dx \\
&&+\left \vert \alpha _{n}-1\right \vert \int \nolimits_{\Omega }\left
\vert u_{n}W^{\prime }(l_{n}u_{n}+(1-l_{n})\widehat{u}_{n})\right \vert dx \\
&\leq &\frac{1}{2}|\alpha _{n}^{2}-1|\int \nolimits_{\Omega }\left \vert
\nabla u_{n}\right \vert ^{2}dx \\
&&+\left \vert \alpha _{n}-1\right \vert \left \{ \int \nolimits_{\Omega }c 
\left[ l_{n}+\left( 1-l_{n}\right) \alpha _{n}\right] ^{p-1}\left \vert
u_{n}\right \vert ^{p}dx\right \} \\
&\rightarrow &0\text{,}
\end{eqnarray*}%
since in the right hand side of the last inequality, the two summands are
products of a zero sequence by a bounded one. Thus, $J(\widehat{u}%
_{n})-J(u_{n})\rightarrow 0$. We return now to%
\begin{equation}
\left \Vert \psi _{n}(0,x)-\widetilde{w}_{n}\right \Vert
_{_{H_{0}^{1}(\Omega )}}\rightarrow 0\text{,}  \label{28}
\end{equation}%
which, as a result of the triangle inequality combined with the boundedness
of $\left \Vert \psi _{n}(0,x)\right \Vert _{_{_{H_{0}^{1}(\Omega )}}}+\left
\Vert w_{n}(x)\right \Vert _{_{_{H_{0}^{1}(\Omega )}}}$, readily gives%
\begin{equation*}
\left \Vert \psi _{n}(0,x)\right \Vert _{_{_{H_{0}^{1}(\Omega )}}}^{2}-\left
\Vert w_{n}(x)\right \Vert _{_{_{H_{0}^{1}(\Omega )}}}^{2}\rightarrow 0\text{%
,}
\end{equation*}%
that is, 
\begin{equation}
\int \nolimits_{\Omega }\left[ \left \vert \nabla u_{n}(0,x)\right \vert
^{2}+u_{n}^{2}(0,x)\left \vert \nabla s_{n}(0,x)\right \vert ^{2}-\left
\vert \nabla w_{n}(x)\right \vert ^{2}\right] dx\rightarrow 0\text{.}
\label{c}
\end{equation}%
We claim that 
\begin{equation}
\int \nolimits_{\Omega }u_{n}^{2}(0,x)\left \vert \nabla s_{n}(0,x)\right
\vert ^{2}dx\rightarrow 0\text{.}  \label{d}
\end{equation}%
If not so, up to a subsequence, 
\begin{equation}
\int \nolimits_{\Omega }\left[ \left \vert \nabla u_{n}(0,x)\right \vert
^{2}-\left \vert \nabla w_{n}(x)\right \vert ^{2}\right] dx\rightarrow k<0%
\text{.}  \label{e}
\end{equation}%
Combining $L^{1}$ convergence of $u_{n}(0,x)-w_{n}(x)$ to $0$, with
Condition 2, we have%
\begin{equation}
\int \nolimits_{\Omega }\left[ W(u_{n}(0,x))-W\left( w_{n}(x)\right) \right]
\rightarrow 0\text{.}  \label{f}
\end{equation}%
Now (\ref{e}) and (\ref{f}) give 
\begin{equation}
J\left( u_{n}(0,x)\right) -J\left( w_{n}(x)\right) \rightarrow k/2<0\text{.}
\label{g}
\end{equation}%
However, as we have shown above, 
\begin{equation}
J(\widehat{u}_{n}(0,x))-J(u_{n}(0,x))\rightarrow 0\text{,}  \label{29}
\end{equation}%
thus obtaining%
\begin{equation}
J(\widehat{u}_{n}(0,x))-J\left( w_{n}(x)\right) \rightarrow k/2<0\text{,}
\label{30}
\end{equation}%
an absurdity, since $\widehat{u}_{n}(0,x)\in S_{\sigma }$, and $w_{n}(x)$ is
a $S_{\sigma }$ minimizer of $J$. Thus (\ref{d}) holds, and because of (\ref%
{c}) and (\ref{f}), we get%
\begin{equation*}
J(u_{n}(0,x))\rightarrow \mu \text{.}
\end{equation*}%
Thus, 
\begin{equation*}
E(\psi _{n}(0,x))-E(\widetilde{w}_{n})\rightarrow 0\text{,}
\end{equation*}%
and by (\ref{27}) we have%
\begin{equation}
E(\psi _{n}(t_{n},x))-E(\widetilde{w}_{n})\rightarrow 0\text{.}  \label{32}
\end{equation}%
From (\ref{32}) we obtain%
\begin{equation}
\int \nolimits_{\Omega }u_{n}^{2}(t_{n},x)\left \vert \nabla
s_{n}(t_{n},x)\right \vert ^{2}dx\rightarrow 0\text{.}  \label{33}
\end{equation}%
To see this, let us assume that this is not the case. Then, up to a
subsequence,%
\begin{equation*}
\int \nolimits_{\Omega }u_{n}^{2}(t_{n},x)\left \vert \nabla
s_{n}(t_{n},x)\right \vert ^{2}dx\rightarrow \rho >0\text{.}
\end{equation*}%
Then (\ref{32}) implies that 
\begin{equation*}
J\left( u_{n}(t_{n},x)\right) -J\left( w_{n}(x)\right) \rightarrow -\rho <0.
\end{equation*}%
Since%
\begin{equation*}
J(\widehat{u}_{n}(t_{n},x))-J(u_{n}(t_{n},x))\rightarrow 0\text{,}
\end{equation*}%
we obtain%
\begin{equation*}
J(\widehat{u}_{n}(t_{n},x))-J\left( w_{n}(x)\right) \rightarrow -\rho <0%
\text{,}
\end{equation*}%
which is absurd, for the same reason as above. That is, (\ref{33}) holds,
resulting to 
\begin{equation*}
J(u_{n})=J(u_{n}(t_{n},x))\rightarrow \mu \text{,}
\end{equation*}%
thus implying 
\begin{equation*}
J(\widehat{u}_{n}(t_{n},x))\rightarrow \mu .
\end{equation*}%
In other words, we may consider $u_{n}=u_{n}(t_{n},x)$ as being a minimizing
sequence in $S_{\sigma }$ for $J(u)$. As such, by the previous discussion, $%
u_{n}(t_{n},x)\rightarrow u^{\prime }(x)$, with $u^{\prime }(x)\in $ $%
S_{\sigma }$, being a minimizer of $J(u)$. Now (\ref{33}) ensures that $%
\left \Vert \psi _{n}(t_{n},x)\right \Vert _{H_{0}^{1}(\Omega )}\rightarrow
\left \Vert u^{\prime }(x)\right \Vert _{H_{0}^{1}(\Omega )}$, and by lower
semicontinuity of the norm we finally obtain that $\psi
_{n}(t_{n},x)\rightarrow u^{\prime }(x)$, thus proving\textit{\ orbital}
stability of the stationary solution.

\begin{remark}
As we claimed before, once $\psi (0,x)$ is sufficiently close to $\Gamma $, $%
\psi (t,x)$, $t\geqq 0$, are all close to $\Gamma $. This implies that the
modulus $\left \vert \psi (t,x)\right \vert $\ is, for any $t\geqq 0$, close
to a (possibly different) ground state of the NLSE. This consequence is the
key tool for proving two main results of the next section, Lemmas 21 and 22.
Thus, the orbital stability in this context becomes a localization result
for the moduli $\left \vert \psi (t,x)\right \vert $.
\end{remark}

We summarize the existence and stability results in the following:

\begin{proposition}
The problem 
\begin{equation*}
ih\frac{\partial \psi }{\partial t}=-\frac{h^{2}}{2}\Delta \psi +\frac{1}{%
2h^{\alpha }}W^{\prime }(\left \vert \psi \right \vert )\frac{\psi }{\left
\vert \psi \right \vert }\text{ in }%
\mathbb{R}
_{0}^{+}\times \Omega \text{, }
\end{equation*}%
\begin{equation}
\psi (0,x)=\phi _{h}(x)\text{, }x\in \Omega \text{,}  \label{36}
\end{equation}%
\begin{equation*}
\psi (t,x)=0\text{ on }%
\mathbb{R}
_{0}^{+}\times \partial \Omega \text{, }
\end{equation*}%
where $h>0$, and $\phi _{h}(x)$ being a suitable initial datum, admits a
stationary solution $\psi (t,x)$ of the form $u_{h}(x)e^{-ikt}$. More
concretely, $u_{h}(x)$ is obtained as a solution of the eigenvalue problem 
\begin{eqnarray}
-\frac{h}{2}\Delta u+\frac{1}{2h^{\alpha +1}}W^{\prime }(u) &=&ku\text{, in }%
\Omega  \notag \\
&&  \label{38} \\
u &\equiv &0\text{, on }\partial \Omega \text{,}  \notag
\end{eqnarray}%
for suitable $k$. In addition, $\psi (t,x)$ is stable in the sense of (\ref%
{r}).
\end{proposition}

We give next the definition of a \textit{solitary wave} with respect to a
bounded domain $\Omega $. To this end, we have to define first the notion of
the \textit{barycenter }of a family of states $\psi _{t}(x)$, $t\geqq 0$,
whose members are obtained by the "time" evolution of an initial state $\psi
_{0}(x)$, in the frame of a proper phase space $X\subset L^{2}(\Omega ,%
\mathbb{C}
)$.

\begin{definition}
For $\psi _{t}(x)$, $t\geqq 0$, as above, its \textit{barycenter, }$q(t)$,%
\textit{\ is defined by the relation }%
\begin{equation}
q(t)=\frac{\int \nolimits_{\Omega }x\left \vert \psi _{t}(x)\right \vert
^{2}dx}{\int \nolimits_{\Omega }\left \vert \psi _{t}(x)\right \vert ^{2}dx}%
\text{.}  \label{34}
\end{equation}
\end{definition}

\begin{remark}
An analogous definition of the \textit{barycenter is given in }\cite%
{Bel-Ben-Bon-Sin}, under the condition that it makes sense. In our case, the
definition of $q(t)$ makes always sense, because $\Omega $ is bounded.
\end{remark}

\begin{remark}
Notice that $q(t)$ does not belong to $\Omega $ necessarily, unless $\Omega $
has specific geometric features. For instance, convexity of $\Omega $ would
ensure that $q(t)\in \Omega $ for all $t\geqq 0$.
\end{remark}

\begin{definition}
The state $\psi \equiv \psi _{0}(x)$ in the phase space $X\subset
L^{2}(\Omega ,%
\mathbb{C}
)$, is called a "solitary wave" in the frame of a dynamical system $%
U_{t}\psi \equiv \psi _{t}(x)$, $t\geqq 0$, where $U:%
\mathbb{R}
_{0}^{+}\times X\rightarrow X$ is the evolution map, if: Given $\varepsilon
>0$, we may find $k(\varepsilon )>0$ such that for each $t\geqq 0$, there
exists a neighborhood $V_{\varepsilon ,t}$ of $q(t)$ with $meas[\Omega
-(V_{\varepsilon ,t}\cap \Omega )]\geqq k(\varepsilon )$, and 
\begin{equation}
\int \nolimits_{\Omega }\left \vert \psi _{t}(x)\right \vert ^{2}dx-\int
\nolimits_{V_{\varepsilon ,t}\cap \Omega }\left \vert \psi _{t}(x)\right
\vert ^{2}dx<\varepsilon \text{.}  \label{35}
\end{equation}
\end{definition}

\begin{remark}
It is easy to see that the stationary solution of (\ref{36}) is a solitary
wave in the above sense: The barycenter in this case is fixed for all $t$,
and one needs to suitably blow up a given neighborhood of it, in order to
meet the requirements of the above definition.
\end{remark}

\section{\protect \bigskip $L^{2}$ localization for $V\neq 0$}

For the rest of the exposition, we assume without loss of generality that $%
0\in \Omega $. We also restrict $p$ in Condition 2 so that $2<p<2+\frac{4}{N}
$, and we impose on $W$ the additional condition:

\begin{condition}
$\exists s_{0\text{ }}$such that $W(s_{0\text{ }})<0$,
\end{condition}

and on $V$ the following one:

\begin{condition}
$V(x)\in C^{0}(\overline{\Omega })$, is nonnegative.
\end{condition}

\begin{remark}
With these restrictions and additional conditions on the nonlinearity and
the potential, problem (\ref{h}) is globally well posed (see \cite[Thm
3.3.1, and Thm 3.4.1]{Caz}), with the energy and the mass remaining constant
in time. Commenting especially on Condition 12, we notice that in the
opposite case, there would exist no nontrvial stationary solution for the
related $%
\mathbb{R}
^{N}$ problem, undermining thus the basic tool we will use in order to prove
the main stability result. In the previous section, where we had posed
weaker conditions on $\Omega $, $W$, and $p$, we were led to nontrivial
stationary solution without any particular comment on $\lambda $. It is
trivial to see that for $W$ having positive lower bound, $\lambda $ is
positive, too. Under the additional conditions imposed above on $\Omega $, $%
W $, and $p$, we can prove, as we will see soon, that we may obtain a
nontrvial solution of (\ref{s}) sitting in $S_{\sigma }$ for suitable $%
\sigma $, with $\lambda =\lambda (\sigma )<0$, provided that $\Omega $
contains a suitably big open ball $B(0,r(\sigma ))$ centered at $0$ with
radius $r(\sigma )$.
\end{remark}

\subsection{Rescalings}

We set $\beta =1+\frac{\alpha }{2}$. For $h<1$, we define the inflated domain%
\begin{equation*}
\Omega _{h}=\left \{ x\in 
\mathbb{R}
^{N}:h^{\beta }x\in \Omega \right \} \text{.}
\end{equation*}%
If $v$ is a $H_{0}^{1}(\Omega _{h})$ solution of the stationary problem 
\begin{eqnarray}
-\Delta u+W^{\prime }(u) &=&\omega u\text{, in }\Omega _{h}  \notag \\
&&  \label{40} \\
u &\equiv &0\text{, on }\partial \Omega _{h}\text{,}  \notag
\end{eqnarray}%
then $v_{h}(x)=v(\frac{x}{h^{\beta }})$ is a $H_{0}^{1}(\Omega )$ solution
of the stationary problem 
\begin{eqnarray}
-h\Delta u+\frac{1}{h^{\alpha +1}}W^{\prime }(u) &=&\frac{\omega }{h^{\alpha
+1}}u\text{, in }\Omega  \notag \\
&&  \label{41} \\
u &\equiv &0\text{, on }\partial \Omega \text{.}  \notag
\end{eqnarray}%
Furthermore, we define the functionals%
\begin{eqnarray*}
C_{h}(u) &=&\frac{1}{h^{N\beta }}\int \nolimits_{\Omega }u^{2}(x)dx\text{, }%
C_{\Omega _{h}}(u)=\int \nolimits_{\Omega _{h}}u^{2}(x)dx\text{,} \\
J_{h}(u) &=&\frac{1}{h^{N\beta }}\int \nolimits_{\Omega }\left[ \frac{h^{2}%
}{2}\left \vert \nabla u\right \vert ^{2}+W_{h}(u)\right] dx\text{, }%
J_{\Omega _{h}}(u)=\int \nolimits_{\Omega _{h}}\left[ \frac{1}{2}\left
\vert \nabla u\right \vert ^{2}+W(u)\right] dx\text{,}
\end{eqnarray*}%
where $W_{h}(u)=\frac{1}{h^{\alpha }}W(u)$.

We have the following identities:%
\begin{equation}
J_{h}(v_{h})=h^{-\alpha }J_{\Omega _{h}}(v)\text{, }C_{h}(v_{h})=C_{\Omega
_{h}}(v)\text{.}  \label{42}
\end{equation}

We next define%
\begin{equation*}
m(h,\Omega ):=\underset{C_{h}=1}{\inf }J_{h}\text{.}
\end{equation*}%
\qquad

\begin{lemma}
For $\Omega _{1}$,$\Omega _{2\text{ }}$two bounded domains as described
above, with $\Omega _{1}\subset \Omega _{2\text{ }}$, it holds $m(h,\Omega
_{2})\leq m(h,\Omega _{1})$. \  \  \  \  \  \  \  \  \  \  \  \  \  \  \  \  \qquad \ 
\end{lemma}

\begin{proof}
It is straightforward, since $H_{0}^{1}(\Omega _{1})\subset H_{0}^{1}(\Omega
_{2\text{ }})$ by extending a function in $H_{0}^{1}(\Omega _{1})$ into a
function in $H_{0}^{1}(\Omega _{2\text{ }})$ by zero on $\Omega _{2\text{ }%
}\backslash \Omega _{1}$.
\end{proof}

\begin{lemma}
$m(h,\Omega )=h^{-\alpha }m(1,\Omega _{h})$, where $m(1,\Omega _{h})=%
\underset{C_{\Omega _{h}}=1}{\inf }J_{\Omega _{h}}$.
\end{lemma}

\begin{proof}
By rescaling.
\end{proof}

Notice that the conditions satisfied by $p$, $W$ correspond to the
prerequisites for the existence result given in \cite{Bel-B-G-M} to hold.
Namely:

\begin{lemma}
There exists some $\overline{\sigma }>0$ such that for all $\sigma >$ $%
\overline{\sigma }$ a positive minimizer $u_{\sigma }\in H^{1}(%
\mathbb{R}
^{N})$ exists for $J(u)$ over all $u\in H^{1}(%
\mathbb{R}
^{N})$, $\left \Vert u\right \Vert _{L^{2}(%
\mathbb{R}
^{N})}=\sigma $. In fact, $\left \Vert u_{\sigma }\right \Vert _{L^{2}(%
\mathbb{R}
^{N})}=\sigma $, $J(u_{\sigma })<0$, and $u_{\sigma }$ is a solution of the $%
\mathbb{R}
^{N}$ version of (\ref{s}), with $\lambda <0$.
\end{lemma}

We consider an increasing sequence $\left \{ r_{n}\right \} $, with $%
r_{n}\rightarrow \infty $. Let%
\begin{equation*}
m_{\sigma }(B(0,r_{n})):=\underset{\left \Vert w\right \Vert
_{L^{2}(B(0,r_{n}))}=\sigma }{\inf }J(w)\text{.}
\end{equation*}%
Trivially, as Lemma 15 indicates, we see that 
\begin{equation}
m_{\sigma }(B(0,r_{1}))\geqq m_{\sigma }(B(0,r_{2}))\geqq \cdot \cdot \cdot
\geqq m_{\sigma }(%
\mathbb{R}
^{N})=\underset{\left \Vert w\right \Vert _{L^{2}(%
\mathbb{R}
^{N})}=\sigma }{\inf }J(w)>-\infty \text{.}  \label{43}
\end{equation}

\begin{lemma}
$\underset{n\rightarrow \infty }{\lim }m_{\sigma }(B(0,r_{n}))=m_{\sigma }(%
\mathbb{R}
^{N})$.
\end{lemma}

\begin{proof}
By Lemma 17, $m_{\sigma }(%
\mathbb{R}
^{N})$ is attained by some $\overline{u}$ $\in $ $H^{1}(%
\mathbb{R}
^{N})$,

with $\left \Vert \overline{u}\right \Vert _{L^{2}(%
\mathbb{R}
^{N})}=\sigma $. Actually $\overline{u}$ is radial (\cite{Caz-Lions}, Thm.
II.1 and Rem. II.3).

For each $n$, we may choose a $C^{\infty }(%
\mathbb{R}
^{N})$ real function $\chi _{n}$ satisfying 
\begin{eqnarray}
\chi _{n} &\equiv &1\text{ if }\left \vert x\right \vert \leq r_{n}/2\text{,}
\notag \\
\chi _{n} &\equiv &0\text{ if }\left \vert x\right \vert \geqq r_{n}\text{,}
\label{44} \\
\left \vert \nabla \chi _{n}\right \vert &\leq &4/r_{n}\text{.}  \notag
\end{eqnarray}%
We define $w_{n}=\chi _{n}\overline{u}$. For suitable $t_{n}>0$, we have $%
\left \Vert t_{n}w_{n}\right \Vert _{L^{2}(%
\mathbb{R}
^{N})}=$

$\left \Vert t_{n}w_{n}\right \Vert _{L^{2}(B(0,r_{n}))}=\sigma $. Setting $%
u_{n}=t_{n}w_{n}$, (\ref{43}) yields%
\begin{equation}
J(u_{n})\geqq m_{\sigma }(B(0,r_{n}))\geqq m_{\sigma }(%
\mathbb{R}
^{N})\text{.}  \label{45}
\end{equation}%
We want to prove $\underset{n\rightarrow \infty }{\lim }J(u_{n})=m_{\sigma }(%
\mathbb{R}
^{N})$, that will finish the proof. We have $w_{n}\rightarrow \overline{u}$
in $L^{2}(%
\mathbb{R}
^{N})$. Thus $t_{n}\rightarrow 1$, and $u_{n}\rightarrow $ $\overline{u}$ in 
$L^{2}(%
\mathbb{R}
^{N})$. We also have%
\begin{eqnarray*}
\int \nolimits_{%
\mathbb{R}
^{N}}\left \vert \nabla w_{n}-\nabla \overline{u}\right \vert ^{2}dx
&=&\int \nolimits_{%
\mathbb{R}
^{N}}\left \vert \left( \nabla \chi _{n}\right) \overline{u}+(\chi
_{n}-1)\nabla \overline{u}\right \vert ^{2}dx \\
&\leq &2\int \nolimits_{%
\mathbb{R}
^{N}}\left \vert \left( \nabla \chi _{n}\right) \overline{u}\right \vert
^{2}dx+2\int \nolimits_{%
\mathbb{R}
^{N}}\left \vert (\chi _{n}-1)\nabla \overline{u}\right \vert ^{2}dx \\
&\leq &32/r_{n}^{2}\int \nolimits_{\left \vert x\right \vert >r_{n}/2}\left
\vert \overline{u}\right \vert ^{2}dx+2\int \nolimits_{\left \vert x\right
\vert >r_{n}/2}\left \vert \nabla \overline{u}\right \vert ^{2}dx\rightarrow
0
\end{eqnarray*}%
as $n\rightarrow \infty $. So $\nabla w_{n}\rightarrow \nabla \overline{u}$
in $L^{2}$, and $\nabla u_{n}\rightarrow \nabla \overline{u}$ in $L^{2}$
too. Thus $u_{n}\rightarrow \overline{u}$ in $H^{1}(%
\mathbb{R}
^{N})$. This is combined with the continuity of the Nemytskii operator%
\begin{equation*}
W:L^{t}\left( 
\mathbb{R}
^{N}\right) \rightarrow L^{1}\left( 
\mathbb{R}
^{N}\right) \text{, }2<t<2^{\ast }\text{,}
\end{equation*}%
to ensure that 
\begin{equation*}
\int \nolimits_{%
\mathbb{R}
^{N}}W(u_{n})dx\rightarrow \int \nolimits_{%
\mathbb{R}
^{N}}W(\overline{u})dx\text{.}
\end{equation*}%
Thus $\underset{n\rightarrow \infty }{\lim }J(u_{n})=m_{\sigma }(%
\mathbb{R}
^{N})$, and the proof has been completed.
\end{proof}

\begin{remark}
The above Lemma makes clear the final assertion of Rem. 14: If $\Omega $
contains a suitably big open ball $B(0,r(\sigma ))$ with $\overline{\sigma }%
<\sigma $, then $m_{\sigma }(1,\Omega )$ has to be negative, and so has to
be the eigenvalue $\lambda =\lambda (\sigma )$ related to (\ref{s}).
\end{remark}

\begin{lemma}
For a sequence of positive numbers $h_{k}\rightarrow 0$, for $k\rightarrow
\infty $, it holds%
\begin{equation*}
\underset{k\rightarrow \infty }{\lim }m_{\sigma }(1,\Omega
_{h_{k}})=m_{\sigma }(%
\mathbb{R}
^{N})\text{,}
\end{equation*}%
where $m_{\sigma }(1,\Omega _{h_{k}})=\underset{C_{\Omega _{h_{k}}}=\sigma }{%
\inf }J_{\Omega _{h_{k}}}$.
\end{lemma}

\begin{proof}
Combine Lemmas 15 and 18.
\end{proof}

\subsection{$L^{2}$ localization}

To facilitate exposition, we make the harmless assumption that $\sigma =1$,
thus suppressing subindices in all involved infima $m$. We have the
following:

\begin{lemma}
For any $\varepsilon >0$, there exist $\delta =\delta (\varepsilon )$, $%
h_{0}=h_{0}(\varepsilon )>0$, and $R=R(\varepsilon )>0$ such that, for any $%
0<h<h_{0}(\varepsilon )$, there is an open ball $B(\widehat{q}_{h},h^{\beta
}R)\subset $ $\Omega $ so that for any $u\in H_{0}^{1}(\Omega )$ with $%
C_{h}(u)=1$, and $J_{h}(u)<m(h,\Omega )+\delta h^{-\alpha }$, 
\begin{equation*}
\frac{1}{h^{N\beta }}\int \nolimits_{\Omega \backslash B(\widehat{q}%
_{h},h^{\beta }R)}u^{2}dx<\varepsilon
\end{equation*}%
to hold.
\end{lemma}

\begin{proof}
We argue by contradiction. Assuming the contrary, there exists $\varepsilon
>0$ such that for any $r>0$ we may find sequences $\delta _{n}=\delta
_{n}(r) $, $h_{n}=h_{n}(r)\rightarrow 0^{+}$, and a sequence $%
u_{h_{n}}=u_{h_{n}}(r)\in H_{0}^{1}(\Omega )$ with $C_{h_{n}}(u_{h_{n}})=1$, 
$J_{h_{n}}(u_{h_{n}})<m(h_{n},\Omega )+\delta _{n}h_{n}^{-\alpha }$ such
that, for all open balls $B(q_{r},h_{n}^{\beta }r)$ $\subset \Omega $, 
\begin{equation*}
\int \nolimits_{\Omega \backslash B(q_{r},h_{n}^{\beta
}r)}u_{h_{n}}^{2}dx\geq \varepsilon h_{n}^{N\beta }
\end{equation*}%
to hold. For each $n$, we now pass to the $\Omega _{h_{n}}$ counterpart of $%
u_{h_{n}}$, denoted by $u_{n}$, that is, $u_{n}(x)=u_{h_{n}}(h_{n}^{\beta
}x) $. Combining (\ref{42}) and Lemma 16, we have $C_{\Omega
_{h_{n}}}(u_{n})=1$, $J_{\Omega _{h_{n}}}(u_{n})<m(1,\Omega _{h_{n}})+\delta
_{n}$, and 
\begin{equation}
\int \nolimits_{\Omega _{h_{n}}\backslash B(h_{n}^{-\beta
}q_{r},r)}u_{n}^{2}dx\geq \varepsilon \text{,}  \label{46}
\end{equation}%
According to Lemma 15 in \cite{Ben-Ghi-Mich}, there is a $\delta >0$, and an
open ball $B(\widehat{q},R^{\prime })$ in $%
\mathbb{R}
^{N}$ such that, for each $w$ in $H^{1}(%
\mathbb{R}
^{N})$, with $\left \Vert w\right \Vert _{L^{2}(%
\mathbb{R}
^{N})}=1$, $J(w)<m(%
\mathbb{R}
^{N})+\delta $, 
\begin{equation*}
\int \nolimits_{%
\mathbb{R}
^{N}\backslash B(\widehat{q},R^{\prime })}w^{2}dx<\varepsilon
\end{equation*}%
to hold. Take now $\delta _{n}(R^{\prime })$, $h_{n}(R^{\prime })$ as above.
Because of Lemma 20, for $n$ big enough, we may ensure that $B(\widehat{q}%
,R^{\prime })\subset \Omega _{h_{n}(R^{\prime })}$, and $m(1,\Omega
_{h_{n}})+\delta _{n}(R^{\prime })<m(%
\mathbb{R}
^{N})+\delta $. Extending $u_{n}$ by $0$ outside $\Omega _{h_{n}(R^{\prime
})}$, we obtain a $H^{1}(%
\mathbb{R}
^{N})$ function meeting the requirements of Lemma 15 in \cite{Ben-Ghi-Mich},
thus%
\begin{equation*}
\int \nolimits_{\Omega _{h_{n}(R^{\prime })}\backslash B(\widehat{q}%
,R^{\prime })}u_{n}^{2}dx=\int \nolimits_{%
\mathbb{R}
^{N}\backslash B(\widehat{q},R^{\prime })}u_{n}^{2}dx<\varepsilon \text{,}
\end{equation*}%
that contradicts (\ref{46}).
\end{proof}

Lemma 21 makes obvious the following:

\begin{lemma}
For any $\varepsilon >0$, there exist $\delta =\delta (\varepsilon )$, $%
h_{0}=h_{0}(\varepsilon )>0$, and $R=R(\varepsilon )>0$ such that, for any $%
0<h<h_{0}(\varepsilon )$, there is an open ball $B(\widehat{q}_{h},h^{\beta
}R)\subset $ $\Omega $ so that for a solution $\psi (t,x)$ of (\ref{h}) with 
$C_{h}(\left \vert \psi (t,x)\right \vert )=1$, and $J_{h}(\left \vert \psi
(t,x)\right \vert )<m(h,\Omega )+\delta h^{-\alpha }$, for each $t\in 
\mathbb{R}
_{0}^{+}$, 
\begin{equation*}
\frac{1}{h^{N\beta }}\int \nolimits_{\Omega \backslash B(\widehat{q}%
_{h},h^{\beta }R)}\left \vert \psi (t,x)\right \vert ^{2}dx<\varepsilon
\end{equation*}%
to hold.
\end{lemma}

The correlation of the solutions of the equations (\ref{40}) and (\ref{41}),
combined with (\ref{42}), ensure the existence of "\textit{ground state}"%
\textit{\ }solutions of (\ref{41}), that is, solutions that are minimizers
of $J_{h}(u)$, with $u$ satisfying $C_{h}(u)=1$.

We define next, the following set of admissible initial data, for given $K$, 
$h>0$:%
\begin{equation}
B_{h}^{K}=\left \{ 
\begin{tabular}{l}
$\psi (0,x)=u_{h}(0,x)e^{\frac{i}{h}s_{h}(0,x)}$ \\ 
with $u_{h}(0,x)=(U+w)(x)$ \\ 
$U$ is a ground state solution of (\ref{41}), and $w\in H_{0}^{1}(\Omega )$
s.t. \\ 
$C_{h}(U+w)=1$, and $\left \Vert w\right \Vert _{H_{0}^{1}(\Omega
)}<Kh^{\alpha }$ \\ 
$\left \Vert \nabla s_{h}(0,x)\right \Vert _{L^{\infty }}\leq Kh^{N\beta /2}$
\\ 
$\int \nolimits_{\Omega }V(x)u_{h}^{2}(0,x)dx\leq Kh^{N\beta }$%
\end{tabular}%
\right \}  \label{47}
\end{equation}%
We prove next the basic stability result.

\begin{proposition}
Given $\varepsilon >0$, there exists $h_{0}=h_{0}(\varepsilon )>0$, and $%
R=R(\varepsilon )>0$ such that, for any $0<h<h_{0}(\varepsilon )$, there is
an open ball $B(\widehat{q}_{h},h^{\beta }R)\subset $ $\Omega $ so that for
a solution $\psi (t,x)$ of (\ref{h}) with $C_{h}(\left \vert \psi
(t,x)\right \vert )=1$, and with initial data $\psi (0,x)\in B_{h}^{K}$,
where $K$ is a positive fixed number, it holds%
\begin{equation*}
\frac{1}{h^{N\beta }}\int \nolimits_{\Omega \backslash B(\widehat{q}%
_{h},h^{\beta }R)}\left \vert \psi (t,x)\right \vert ^{2}dx<\varepsilon 
\text{,}
\end{equation*}%
for any $t\in 
\mathbb{R}
_{0}^{+}$.
\end{proposition}

\begin{proof}
Because of conservation of energy, we have%
\begin{eqnarray*}
E(\psi (t,x)) &=&E(\psi (0,x)) \\
&=&h^{N\beta }J_{h}(u_{h}(0,x))+\int \nolimits_{\Omega }u_{h}^{2}(0,x)\left[ 
\frac{\left \vert \nabla s_{h}(0,x)\right \vert ^{2}}{2}+V(x)\right] dx \\
&\leq &h^{N\beta }J_{h}(u_{h}(0,x))+\frac{K^{2}}{2}h^{N\beta }+Kh^{N\beta }
\\
&=&h^{N\beta }J_{h}(u_{h}(0,x))+Ch^{N\beta } \\
&=&h^{N\beta }J_{h}(U+w)+Ch^{N\beta } \\
&\leq &h^{N\beta }(m(h,\Omega )+C^{\prime }Kh^{\alpha }+C) \\
&=&h^{N\beta }\left[ m(h,\Omega )+h^{-\alpha }\left( h^{2\alpha }C^{\prime
}K+Ch^{\alpha }\right) \right] \text{, }
\end{eqnarray*}%
since a Mean Value Theorem application ensures that $J_{h}(U+w)\leq
m(h,\Omega )+C^{\prime }Kh^{\alpha }$, for suitably small $h$. More
precisely, since $J_{h}$ is $C^{1}$, we may find some $\eta \in (0,1)$ so
that%
\begin{eqnarray*}
J_{h}(U+w)-J_{h}(U) &=&J_{h}^{\prime }(U+\eta w)\left[ w\right] \\
&=&\int \nolimits_{\Omega }\nabla (U+\eta w)\nabla wdx \\
&&+\int \nolimits_{\Omega }W^{\prime }(U+\eta w)wdx\text{.}
\end{eqnarray*}%
We have%
\begin{equation}
\left \vert \int \nolimits_{\Omega }\nabla (U+\eta w)\nabla wdx\right \vert
\leq \left \Vert \nabla U\right \Vert _{L^{2}(\Omega )}\left \Vert \nabla
w\right \Vert _{L^{2}(\Omega )}+\left \Vert \nabla w\right \Vert
_{L^{2}(\Omega )}^{2}\leq C_{1}\left \Vert w\right \Vert _{H_{0}^{1}(\Omega
)}\text{,}  \label{48}
\end{equation}%
since $\left \Vert w\right \Vert _{H_{0}^{1}(\Omega )}\leq 1$. Because of
Condition 2, 
\begin{equation}
\left \vert \int \nolimits_{\Omega }W^{\prime }(U+\eta w)wdx\right \vert
\leq C_{2}\left \Vert w\right \Vert _{L^{p}(\Omega )}\leq C_{3}\left \Vert
w\right \Vert _{H_{0}^{1}(\Omega )}\text{.}  \label{49}
\end{equation}%
Combining (\ref{48}) and (\ref{49}), we obtain the desired inequality. Thus%
\begin{eqnarray*}
J_{h}(u_{h}(t,x)) &=&h^{-N\beta }\left \{ E(\psi (t,x))-\int
\nolimits_{\Omega }u_{h}^{2}(t,x)\left[ \frac{\left \vert \nabla
s_{h}(t,x)\right \vert ^{2}}{2}+V(x)\right] dx\right \} \\
&\leq &\left[ m(h,\Omega )+h^{-\alpha }\left( h^{2\alpha }C^{\prime
}K+Ch^{\alpha }\right) \right] \text{,}
\end{eqnarray*}%
since $V(x)\geq 0$. We use now Lemma 22, by choosing $h_{0}$ small enough in
order to ensure $h_{0}^{2\alpha }C^{\prime }K+Ch_{0}^{\alpha }\leq \delta
(\varepsilon )$, and the result follows.
\end{proof}

\begin{corollary}
Given $\varepsilon >0$, if $q(t)$ is the barycenter of a solution $\psi
(t,x) $ of (\ref{h}) with $C_{h}(\left \vert \psi (t,x)\right \vert )=1$,
and with initial data $\psi (0,x)\in B_{h}^{K}$, where $K$ is a positive
fixed number, then for any $t\in 
\mathbb{R}
_{0}^{+}$, $q(t)\in B(\widehat{q}_{h},d\varepsilon +h^{\beta }R)$, where $h$
and $R$ are as in Proposition 23, and $d=diam(\Omega )$.
\end{corollary}

\begin{proof}
One has for any $t\in 
\mathbb{R}
_{0}^{+}$, 
\begin{eqnarray*}
\left \vert \widehat{q}_{h}-q(t)\right \vert &=&\left \vert \frac{\int
\nolimits_{\Omega }\left( \widehat{q}_{h}-x\right) \left \vert \psi
(t,x)\right \vert ^{2}dx}{\int \nolimits_{\Omega }\left \vert \psi
(t,x)\right \vert ^{2}dx}\right \vert =h^{-N\beta }\left \vert \int
\nolimits_{\Omega }\left( \widehat{q}_{h}-x\right) \left \vert \psi
(t,x)\right \vert ^{2}dx\right \vert \\
&\leq &h^{-N\beta }\int \nolimits_{\Omega }\left \vert \widehat{q}%
_{h}-x\right \vert \left \vert \psi (t,x)\right \vert ^{2}dx \\
&=&h^{-N\beta }\left[ 
\begin{array}{c}
\int \nolimits_{\Omega \backslash B(\widehat{q}_{h},h^{\beta }R)}\left
\vert \widehat{q}_{h}-x\right \vert \left \vert \psi (t,x)\right \vert ^{2}dx
\\ 
+\int \nolimits_{B(\widehat{q}_{h},h^{\beta }R)}\left \vert \widehat{q}%
_{h}-x\right \vert \left \vert \psi (t,x)\right \vert ^{2}dx%
\end{array}%
\right] \\
&<&h^{-N\beta }\left[ dh^{N\beta }\varepsilon +h^{\beta }Rh^{N\beta }\right]
=d\varepsilon +h^{\beta }R\text{,}
\end{eqnarray*}%
thus proving the claim.

\begin{remark}
The above Corollary may be interpreted as follows: If we choose a pretty
small $\varepsilon >0$, and we make a consequent choice of small $h$, then $%
B(\widehat{q}_{h},h^{\beta }R)\subset A_{h}=B(\widehat{q}_{h},d\varepsilon
+h^{\beta }R)\cap \Omega $, with $meas(\Omega -A_{h})>0$. In other words, we
have found a neighborhood of $q(t)$, \textbf{for any }$t\in 
\mathbb{R}
_{0}^{+}$, as in Def. 10, in the sense that a solution $\psi (t,x)$ of (\ref%
{h}) with a perturbed initial state, meeting the requirements of Proposition
23, concentrates basically on this neighborhood of $q(t)$, exhibiting a
behavior quite similar to a solitary wave. \ 
\end{remark}
\end{proof}

\section{$H^{1}$\ localization for $V\neq 0$}

\subsection{The case\ $\Omega =%
\mathbb{R}
^{N}$\  \  \ }

\  \ We start with the following assumption: The problem 
\begin{equation*}
ih\frac{\partial \psi }{\partial t}=-\frac{h^{2}}{2}\Delta \psi +\frac{1}{%
2h^{\alpha }}W^{\prime }(\left \vert \psi \right \vert )\frac{\psi }{\left
\vert \psi \right \vert }+V(x)\psi \text{ ,}
\end{equation*}%
\begin{equation}
\psi (0,x)=\phi _{h}(x)\text{,}  \label{ph}
\end{equation}%
\begin{equation*}
\left \vert \psi (t,x)\right \vert \rightarrow 0\text{, as }\left \vert
x\right \vert \rightarrow \infty \text{,}
\end{equation*}%
considered as the $%
\mathbb{R}
^{N}$ analogue of (\ref{h}) admits a unique solution%
\begin{equation}
\psi \in C^{0}(%
\mathbb{R}
,\text{ }H^{2}(%
\mathbb{R}
^{N}))\cap C^{1}(%
\mathbb{R}
,L^{2}(%
\mathbb{R}
^{N}))  \label{c1}
\end{equation}%
(see \cite{Ka}, \cite{Caz} or \cite{Gi-Ve} for sufficient conditions.)\ We
also impose on $W$, $V$ further conditions, namely:

\begin{condition}
$W$ is $C^{3}$, with $\left \vert W^{\prime \prime }\right \vert \leq c\left
\vert s\right \vert ^{p-2}$\ for some $c\geq 0$, $2<p<2+\frac{4}{N}$\  \ 
\end{condition}

\begin{condition}
$V:$ $%
\mathbb{R}
^{N}\rightarrow 
\mathbb{R}
$ is a $C^{0}$ nonnegative\ function
\end{condition}

In order to proceed, we will need the following lemma:

\begin{lemma}
For every $\varepsilon >0$, there exists $R=R(\varepsilon )>0$ such that for
every ground state $U$, there exists $q(U)\in 
\mathbb{R}
^{N}$ such that 
\begin{equation*}
\int \nolimits_{%
\mathbb{R}
^{N}\backslash B(q(U),R)}\left( \left \vert \nabla U\right \vert
^{2}+U^{2}\right) dx<\varepsilon
\end{equation*}
\end{lemma}

\begin{proof}
If we assume the contrary, then we may find $\varepsilon >0$ such that we
may have a sequence of pairs $(R_{n}>0,U_{n}$ ground state$)$ so that for
each $q\in 
\mathbb{R}
^{N}$ 
\begin{equation*}
\int \nolimits_{%
\mathbb{R}
^{N}\backslash B(q(U),R_{n})}\left( \left \vert \nabla U_{n}\right \vert
^{2}+U_{n}^{2}\right) dx\geq \varepsilon \text{,}
\end{equation*}%
thus obtaining%
\begin{equation}
\underset{q\in 
\mathbb{R}
^{N}}{\inf }\int \nolimits_{%
\mathbb{R}
^{N}\backslash B(q(U),R_{n})}\left( \left \vert \nabla U_{n}\right \vert
^{2}+U_{n}^{2}\right) dx\geq \varepsilon \text{.}  \label{p1}
\end{equation}%
Then $\left \{ U_{n}\right \} _{n}$ is a minimizing sequence, and by
concentration compactness we know that $\left \{ U_{n}\right \} _{n}$ is
relatively compact up to a translation by $\left \{ q_{n}\right \} _{n}$ $%
\in 
\mathbb{R}
^{N}$. Thus there exists a ground state $U$ with $U_{n}(\cdot -$ $%
q_{n})\rightarrow U$ in $H^{1}(%
\mathbb{R}
^{N})$, and%
\begin{eqnarray*}
&&\underset{q\in 
\mathbb{R}
^{N}}{\inf }\int \nolimits_{%
\mathbb{R}
^{N}\backslash B(q(U),R_{n})}\left( \left \vert \nabla U_{n}\right \vert
^{2}+U_{n}^{2}\right) dx \\
&\leq &\int \nolimits_{%
\mathbb{R}
^{N}\backslash B(-q_{n},R_{n})}\left( \left \vert \nabla U_{n}\right \vert
^{2}+U_{n}^{2}\right) dx \\
&=&\int \nolimits_{%
\mathbb{R}
^{N}\backslash B(0,R_{n})}(\left \vert \nabla U_{n}\right \vert
^{2}+U_{n}^{2})(x-q_{n})dx \\
&=&\int \nolimits_{%
\mathbb{R}
^{N}\backslash B(0,R_{n})}\left( \left \vert \nabla U\right \vert
^{2}+U^{2}\right) dx+o_{n}(1)=o_{n}(1)\text{,}
\end{eqnarray*}%
contradicting (\ref{p1}).
\end{proof}

\begin{lemma}
For every $\varepsilon >0$, there exist $\widehat{R}=\widehat{R}(\varepsilon
)>0$, $\delta $ $=\delta (\varepsilon )>0$ such that, for any $u\in
J^{m+\delta }\cap S_{\sigma }$, we can find a point $\widehat{q}=\widehat{q}%
(u)\in 
\mathbb{R}
^{N}$ such that%
\begin{equation}
\frac{\int \nolimits_{%
\mathbb{R}
^{N}\backslash B(\widehat{q},\widehat{R})}\left \vert \nabla u\right \vert
^{2}(x)dx}{\int \nolimits_{%
\mathbb{R}
^{N}}\left \vert \nabla u\right \vert ^{2}(x)dx}<\varepsilon \text{,}
\label{p2}
\end{equation}%
where $m=m(%
\mathbb{R}
^{N})$ (see Lemma 21), $J^{m+\delta }=\left \{ u\in H^{1}(%
\mathbb{R}
^{N})/J(u)<m+\delta \right \} $.
\end{lemma}

\begin{proof}
Exploiting Rel. (50) in \cite{Ben-Ghi-Mich}, we obtain a point $\widehat{q}=%
\widehat{q}(u)\in 
\mathbb{R}
^{N}$ and a radial ground state solution $U$ such that 
\begin{equation}
u(x)=U(x-\widehat{q})+w\text{, with }\left \Vert w\right \Vert _{H^{1}(%
\mathbb{R}
^{N})}\leq C\varepsilon \text{,}  \label{p3}
\end{equation}%
where $C$ is a constant not depending on $U$. According to the previous
Lemma, we may find $R>0$ and a point $q=q(U)$ such that\  \  \  \  \  \  \  \  \  \  \
\  \  \  \  \  \  \  \  \  \  \  \  \  \  \  \  \  \  \  \  \  \  \  \  \  \  \  \  \  \  \  \  \  \  \  \  \  \
\  \  \  \  \  \  \  \  \  \  \  \  \  \  \  \  \  \  \  \  \  \  \  \  \  \  \  \  \  \  \  \  \  \  \  \  \  \
\  \  \  \  \  \  \  \  \  \  \  \  \  \  \  \ 
\begin{equation}
\int \nolimits_{%
\mathbb{R}
^{N}\backslash B(q,R)}\left( \left \vert \nabla U\right \vert
^{2}+U^{2}\right) dx<\frac{\sigma ^{2}C\varepsilon }{c_{1}^{2}}\text{,}
\label{p4}
\end{equation}%
where $c_{1}$ is the Sobolev constant related to the embedding $H^{1}(%
\mathbb{R}
^{N})\rightarrow L^{2}(%
\mathbb{R}
^{N})$. If we choose $\widehat{R}$ big enough, then $B(q,R)\subset B(0,%
\widehat{R})$, resulting to 
\begin{equation}
\int \nolimits_{%
\mathbb{R}
^{N}\backslash B(0,\widehat{R})}\left( \left \vert \nabla U\right \vert
^{2}+U^{2}\right) dx<\frac{\sigma ^{2}C\varepsilon }{c_{1}^{2}}\text{.}
\label{p5}
\end{equation}%
We have%
\begin{equation}
\frac{\int \nolimits_{%
\mathbb{R}
^{N}\backslash B(0,\widehat{R})}\left \vert \nabla U\right \vert ^{2}(x)dx}{%
\int \nolimits_{%
\mathbb{R}
^{N}}\left \vert \nabla U\right \vert ^{2}(x)dx}<\frac{c_{1}^{2}\int
\nolimits_{%
\mathbb{R}
^{N}\backslash B(0,\widehat{R})}\left( \left \vert \nabla U\right \vert
^{2}+U^{2}\right) dx}{\int \nolimits_{%
\mathbb{R}
^{N}}U^{2}dx}<C\varepsilon \text{.}  \label{p6}
\end{equation}%
Now%
\begin{eqnarray}
\int \nolimits_{%
\mathbb{R}
^{N}\backslash B(\widehat{q},\widehat{R})}\left \vert \nabla u\right \vert
^{2}(x)dx &<&\int \nolimits_{%
\mathbb{R}
^{N}\backslash B(\widehat{q},\widehat{R})}\left \vert \nabla U\right \vert
^{2}(x-\widehat{q})dx  \notag \\
&&+\int \nolimits_{%
\mathbb{R}
^{N}\backslash B(\widehat{q},\widehat{R})}\left( \left \vert \nabla w\right
\vert ^{2}+2wU\right) dx  \notag \\
&=&\int \nolimits_{%
\mathbb{R}
^{N}\backslash B(0,\widehat{R})}\left \vert \nabla U\right \vert ^{2}(x)dx 
\notag \\
&&+\int \nolimits_{%
\mathbb{R}
^{N}\backslash B(\widehat{q},\widehat{R})}\left( \left \vert \nabla w\right
\vert ^{2}+2\nabla w\nabla U\right) dx\text{.}  \label{p8}
\end{eqnarray}%
By (\ref{p3}), (\ref{p6}) and (\ref{p8}), we get the claim. One should
notice that $\widehat{R}$ does not depend on $u$, $U$.\ 
\end{proof}

Our main objective in this subsection is to prove an $H^{1}$ modular
stability result of the solution of (\ref{ph}) with suitable initial data;
more precisely, we prove that, for fixed $t\in 
\mathbb{R}
_{0}^{+}$, this solution is a function on $%
\mathbb{R}
^{N}$\ with one peak localized in a ball with its center depending on $t$
whereas its radius not. To this end, it is sufficient to assume that (\ref%
{ph}) admits global solutions $\psi (t,x)\in C(%
\mathbb{R}
,H^{1}(%
\mathbb{R}
^{N}))$ satisfying the conservation of the energy and of the $L^{2}$ norm.

Given $K>0$, $h>0$, we define the following set of admissible data: 
\begin{equation}
B_{K,h}=\left \{ 
\begin{tabular}{l}
$\psi (0,x)=u_{h}(0,x)e^{\frac{i}{h}s_{h}(0,x)}$ \\ 
with $u_{h}(0,x)=(U+w)\left( \frac{x-q}{h^{\beta }}\right) $ \\ 
$q\in 
\mathbb{R}
^{N}$, $U$ is a ground state solution, and $w\in H^{1}(%
\mathbb{R}
^{N})$ s.t. \\ 
$\left \Vert U+w\right \Vert _{L^{2}}=\left \Vert U\right \Vert
_{L^{2}}=\sigma $, and $\left \Vert w\right \Vert _{H^{1}}<Kh^{\alpha }$ \\ 
$\left \Vert \nabla s_{h}(0,x)\right \Vert _{L^{\infty }}\leq K$ for all $h$
\\ 
$\int \nolimits_{%
\mathbb{R}
^{N}}V(x)u_{h}^{2}(0,x)dx\leq Kh^{N\beta -2\alpha }$%
\end{tabular}%
\right \} \text{.}  \label{p9}
\end{equation}

We next study the rescaling properties of the internal energy 
\begin{equation*}
\widetilde{J}_{h}(u)=\int \nolimits_{%
\mathbb{R}
^{N}}\left( \frac{h^{2}}{2}\left \vert \nabla u\right \vert ^{2}+\frac{1}{%
h^{\alpha }}W(u)\right) dx\text{,}
\end{equation*}
and of the $L^{2}$ norm of a function $u(x)$ having the form%
\begin{equation*}
u(x)=v\left( \frac{x}{h^{\beta }}\right) \text{,}
\end{equation*}%
with $\beta =1+\frac{\alpha }{2}$. We have%
\begin{equation*}
\left \Vert u\right \Vert _{L^{2}}^{2}=\int \nolimits_{%
\mathbb{R}
^{N}}v\left( \frac{x}{h^{\beta }}\right) ^{2}dx=h^{N\beta }\int \nolimits_{%
\mathbb{R}
^{N}}v(\xi )^{2}d\xi =h^{N\beta }\left \Vert v\right \Vert _{L^{2}}^{2}\text{%
,}
\end{equation*}%
and%
\begin{eqnarray*}
\widetilde{J}_{h}(u) &=&\int \nolimits_{%
\mathbb{R}
^{N}}\left[ \frac{h^{2}}{2}\left \vert \nabla _{x}v\left( \frac{x}{h^{\beta }%
}\right) \right \vert ^{2}+\frac{1}{h^{\alpha }}W\left( v\left( \frac{x}{%
h^{\beta }}\right) \right) \right] dx \\
&=&\int \nolimits_{%
\mathbb{R}
^{N}}\left[ \frac{h^{(N-2)\beta +2}}{2}\left \vert \nabla _{\xi }v\left( \xi
\right) \right \vert ^{2}+h^{N\beta -\alpha }W\left( v\left( \xi \right)
\right) \right] d\xi \\
&=&h^{N\beta -\alpha }\int \nolimits_{%
\mathbb{R}
^{N}}\frac{1}{2}\left[ \left \vert \nabla _{\xi }v\left( \xi \right) \right
\vert ^{2}+W\left( v\left( \xi \right) \right) \right] =h^{N\beta -\alpha
}J(v)\text{.}
\end{eqnarray*}

We can now describe the concentration properties of the modulus of the
solution of (\ref{ph}).

\begin{lemma}
For any $\varepsilon >0$, there exist positive numbers $\delta =\delta
(\varepsilon )$, $\widehat{R}=\widehat{R}(\varepsilon )$ such that: for any $%
\psi (t,x)$ that solves (\ref{ph}), with $\left \vert \psi (t,h^{\beta
}x)\right \vert $ $\in J^{m+\delta }\cap S_{\sigma }$, for all $t$, there
exists a map $\widehat{q}_{h}:%
\mathbb{R}
_{0}^{+}\rightarrow $\ $%
\mathbb{R}
^{N}$ for which%
\begin{equation}
\frac{\int \nolimits_{%
\mathbb{R}
^{N}\backslash B(\widehat{q}_{h}(t),\widehat{R})}\left \vert \nabla
u_{h}(t,x)\right \vert ^{2}dx}{\int \nolimits_{%
\mathbb{R}
^{N}}\left \vert \nabla u_{h}(t,x)\right \vert ^{2}dx}<\varepsilon \text{.}
\label{p10}
\end{equation}
\end{lemma}

\begin{proof}
For fixed $h$ and $t$, we set $v(\xi )=\left \vert \psi (t,h^{\beta }\xi
)\right \vert $.\ By Lemma 29, there exist $\delta >0$, $\widehat{R}>0$ and $%
\overline{q}=\overline{q}(v)$ such that: if $\left \vert \psi (t,h^{\beta
}x)\right \vert $ $\in J^{m+\delta }\cap S_{\sigma }$, then 
\begin{equation*}
\varepsilon >\frac{\int \nolimits_{%
\mathbb{R}
^{N}\backslash B(\overline{q},\widehat{R})}\left \vert \nabla _{\xi }v\left(
\xi \right) \right \vert ^{2}d\xi }{\int \nolimits_{%
\mathbb{R}
^{N}}\left \vert \nabla _{\xi }v\left( \xi \right) \right \vert ^{2}d\xi }%
\text{.}
\end{equation*}%
By a change of variables, we obtain%
\begin{eqnarray*}
\varepsilon &>&\frac{\int \nolimits_{%
\mathbb{R}
^{N}\backslash B(\overline{q},\widehat{R})}\left \vert \nabla _{\xi }v\left(
\xi \right) \right \vert ^{2}d\xi }{\int \nolimits_{%
\mathbb{R}
^{N}}\left \vert \nabla _{\xi }v\left( \xi \right) \right \vert ^{2}d\xi }=%
\frac{\int \nolimits_{%
\mathbb{R}
^{N}\backslash B(h^{\beta }\overline{q},h^{\beta }\widehat{R})}\left \vert
\nabla u_{h}(t,x)\right \vert ^{2}dx}{\int \nolimits_{%
\mathbb{R}
^{N}}\left \vert \nabla u_{h}(t,x)\right \vert ^{2}dx} \\
&>&\frac{\int \nolimits_{%
\mathbb{R}
^{N}\backslash B(h^{\beta }\overline{q},\widehat{R})}\left \vert \nabla
u_{h}(t,x)\right \vert ^{2}dx}{\int \nolimits_{%
\mathbb{R}
^{N}}\left \vert \nabla u_{h}(t,x)\right \vert ^{2}dx}\text{, since }h<1%
\text{.}
\end{eqnarray*}%
Setting $\widehat{q}_{h}(t)=$\ $h^{\beta }\overline{q}$, we complete the
proof. Notice that $\widehat{q}_{h}(t)$depends on $\varepsilon $, and $\psi $%
, while $\widehat{R}$ depends only on $\varepsilon $.
\end{proof}

\begin{proposition}
Let $V\in L_{loc}^{\infty }$. For every $\varepsilon >0$, there exists $%
\widehat{R}>0$\ and $h_{0}>0$ such that, for any $\psi (t,x)$ that solves (%
\ref{ph}), with initial data $\psi (0,x)$ $\in $\ $B_{K,h}$, where $h<$\ $%
h_{0}$, and for any $t$, there exists $\widehat{q}_{h}(t)\in $\ $%
\mathbb{R}
^{N}$, for which%
\begin{equation}
\frac{1}{\left \Vert \nabla u_{h}(t,x)\right \Vert _{L^{2}}^{2}}\int
\nolimits_{%
\mathbb{R}
^{N}\backslash B(\widehat{q}_{h}(t),h^{\beta }\widehat{R})}\left \vert
\nabla u_{h}(t,x)\right \vert ^{2}dx<\varepsilon \text{.}  \label{p11}
\end{equation}
\end{proposition}

\begin{proof}
By the conservation law, the energy $E_{h}(\psi (t,x))$ is constant with
respect to $t$. Then we have%
\begin{eqnarray*}
E_{h}(\psi (t,x)) &=&E_{h}(\psi (0,x)) \\
&=&\widetilde{J}_{h}(u_{h}(0,x))+\int \nolimits_{%
\mathbb{R}
^{N}}u_{h}^{2}(0,x)\left[ \frac{\left \vert \nabla s_{h}(0,x)\right \vert
^{2}}{2}+V(x)\right] dx \\
&\leq &\widetilde{J}_{h}(u_{h}(0,x))+\frac{K}{2}\sigma ^{2}h^{N\beta
}+Kh^{N\beta } \\
&&h^{N\beta -\alpha }J(U+w)+h^{N\beta }C\text{,}
\end{eqnarray*}%
where $C$ is a suitable constant. By rescaling , and using that $\psi
(0,x)\in B_{K,h}$, and that $\left \Vert w\right \Vert _{H^{1}}<Kh^{\alpha }$
implies $J(U+w)<m+Kh^{\alpha }$ (see the proof of Proposition 23), we obtain%
\begin{eqnarray*}
E_{h}(\psi (t,x)) &=&h^{N\beta -\alpha }J(U+w)+h^{N\beta }C \\
&<&h^{N\beta -\alpha }\left( m+Kh^{\alpha }\right) +h^{N\beta }C \\
&=&h^{N\beta -\alpha }\left( m+Kh^{\alpha }+h^{\alpha }C\right) =h^{N\beta
-\alpha }\left( m+h^{\alpha }C_{1}\right) \text{,}
\end{eqnarray*}%
where we have set $C_{1}=K+C$. Thus 
\begin{eqnarray}
\widetilde{J}_{h}(u_{h}(t,x)) &=&E_{h}(\psi (t,x))-\int \nolimits_{%
\mathbb{R}
^{N}}u_{h}^{2}(t,x)\left[ \frac{\left \vert \nabla s_{h}(t,x)\right \vert
^{2}}{2}+V(x)\right] dx  \notag \\
&<&h^{N\beta -\alpha }\left( m+h^{\alpha }C_{1}\right) \text{,}  \label{p12}
\end{eqnarray}%
since $V\geq 0$. By rescaling the inequality (\ref{p12}), we get%
\begin{equation*}
J\left( u_{h}(t,h^{\beta }x)\right) <m+h^{\alpha }C_{1}\text{.}
\end{equation*}%
So, for $h$ sufficiently small, we may apply Lemma 30, and get the claim.
\end{proof}

\subsection{The case $\Omega $ is bounded}

The case where $\Omega $ is bounded is easily treated, once one makes the
crucial remark that Lemma 29 has to replace Lemma 15 in \cite{Ben-Ghi-Mich},
that it was used in the proof of Lemma 21. The rest of the proofs in the
consequent Lemmas 32 and 33 and of the final Proposition 34 follow precisely
the pattern of the proofs for Lemmas 21, 22, and of Proposition 23,
respectively. For completeness, we give below the precise statements, where
we have assumed for simplicity, as in the $L^{2}$ case, that $\sigma =1$.

\begin{lemma}
For any $\varepsilon >0$, there exist $\delta =\delta (\varepsilon )$, $%
h_{0}=h_{0}(\varepsilon )>0$, and $R=R(\varepsilon )>0$ such that, for any $%
0<h<h_{0}(\varepsilon )$, there is an open ball $B(\widehat{q}_{h},h^{\beta
}R)\subset $ $\Omega $ so that for any $u\in H_{0}^{1}(\Omega )$ with $%
C_{h}(u)=1$, and $J_{h}(u)<m(h,\Omega )+\delta h^{-\alpha }$, 
\begin{equation*}
\frac{\int \nolimits_{\Omega \backslash B(\widehat{q}_{h},h^{\beta
}R)}\left \vert \nabla u\right \vert ^{2}dx}{\int \nolimits_{\Omega }\left
\vert \nabla u\right \vert ^{2}dx}<\varepsilon
\end{equation*}%
to hold.
\end{lemma}

\begin{lemma}
For any $\varepsilon >0$, there exist $\delta =\delta (\varepsilon )$, $%
h_{0}=h_{0}(\varepsilon )>0$, and $R=R(\varepsilon )>0$ such that, for any $%
0<h<h_{0}(\varepsilon )$, there is an open ball $B(\widehat{q}_{h},h^{\beta
}R)\subset $ $\Omega $ so that for a solution $\psi (t,x)$ of (\ref{h}) with 
$C_{h}(\left \vert \psi (t,x)\right \vert )=1$, and $J_{h}(\left \vert \psi
(t,x)\right \vert )<m(h,\Omega )+\delta h^{-\alpha }$, for each $t\in 
\mathbb{R}
_{0}^{+}$, 
\begin{equation*}
\frac{\int \nolimits_{\Omega \backslash B(\widehat{q}_{h},h^{\beta
}R)}\left \vert \nabla u(t,x)\right \vert ^{2}dx}{\int \nolimits_{\Omega
}\left \vert \nabla u(t,x)\right \vert ^{2}dx}<\varepsilon
\end{equation*}%
to hold, where $u(t,x)=\left \vert \psi (t,x)\right \vert $.
\end{lemma}

\begin{proposition}
Given $\varepsilon >0$, there exists $h_{0}=h_{0}(\varepsilon )>0$, and $%
R=R(\varepsilon )>0$ such that, for any $0<h<h_{0}(\varepsilon )$, there is
an open ball $B(\widehat{q}_{h},h^{\beta }R)\subset $ $\Omega $ so that for
a solution $\psi (t,x)$ of (\ref{h}) with $C_{h}(\left \vert \psi
(t,x)\right \vert )=1$, and with initial data $\psi (0,x)\in B_{h}^{K}$,
where $K$ is a positive fixed number, it holds%
\begin{equation*}
\frac{\int \nolimits_{\Omega \backslash B(\widehat{q}_{h},h^{\beta
}R)}\left \vert \nabla u(t,x)\right \vert ^{2}dx}{\int \nolimits_{\Omega
}\left \vert \nabla u(t,x)\right \vert ^{2}dx}<\varepsilon \text{,}
\end{equation*}%
for any $t\in 
\mathbb{R}
_{0}^{+}$, where $u(t,x)=\left \vert \psi (t,x)\right \vert $.
\end{proposition}

\section{\protect \bigskip Appendix}

\  \ In order to gain some control over the dynamics of the problem, that is,
to try to formulate Newton's equation describing the motion of the
barycenter for a fixed value of $h$, one needs to express suitably $\overset{%
..}{q}(t)$. To this end, a further assumption on $W$ is made, namely that $%
W(0)=0$. For the sake of simplicity, we fix $h=1$, $\alpha =1$ in what
follows, the general case being straightforward. Also, we suppose that a
solution $\psi (t,x)$ is sufficiently smooth in order to make the
integration by parts meaningful. Given this, the general case can be proved
with minor technical efforts. Finally, we use the Einstein convention on the
summation indices.

We will use the Lagrangian formalism. Equation (\ref{h}) is the
Euler-Lagrange equation relative to the following Lagrangian density $%
\mathcal{L}$: 
\begin{equation*}
\mathcal{L}=\mathrm{Re}(i\overline{\psi }\partial _{t}\psi )-\frac{1}{2}\left
\vert \nabla \psi \right \vert ^{2}-W(|\psi |)-V(x)\left \vert \psi \right
\vert ^{2}
\end{equation*}

By Noether's theorem, there are continuity equations related to $\mathcal{L}$%
, which we will use to derive an equation for the motion. In particular, we
are interested in the following continuity equations:

\begin{equation}
\frac{d}{dt}|\psi (t,x)|^{2}=-\nabla \cdot \mathrm{Im}\left( \overline{\psi }%
\nabla \psi )\right)  \label{C}
\end{equation}%
and 
\begin{equation}
\frac{d}{dt}\mathrm{Im}\left( \overline{\psi }\nabla \psi )\right) =-|\psi
|^{2}\nabla V-\nabla \cdot T\text{,}  \label{T}
\end{equation}%
where $T$ is the so called energy stress tensor and has the form 
\begin{equation*}
T_{jk}=\mathrm{Re}\left( \partial _{x_{j}}\psi \partial _{x_{k}}\overline{\psi 
}\right) -\delta _{jk}\left[ \mathrm{Re}\left( \frac{1}{2}\overline{\psi }%
\Delta \psi \right) +\frac{1}{2}\left \vert \nabla \psi \right \vert ^{2}-%
\frac{1}{2}W^{\prime }\left( |\psi |\right) \left \vert \psi \right \vert
+W(|\psi |)\right]
\end{equation*}%
For an introduction to the Lagrangian formalism for equation (\ref{h}) and
continuity equations we refer to \cite{Benci,Ben-Ghi-Mich-arma,Gel-Fom}.

In the light of equation \ref{C}, and by divergence theorem, one has for $%
j=1,...,N$, 
\begin{align*}
\overset{.}{q}_{j}(t)& =\frac{d}{dt}\int_{\Omega }x_{j}\left \vert \psi
(t,x)\right \vert ^{2}dx=\int_{\Omega }x_{j}\partial _{t}(|\psi |^{2})dx \\
& =-\int_{\Omega }x_{j}\nabla \cdot \mathrm{Im}\left( \overline{\psi }\nabla
\psi \right) dx \\
& =-\int_{\Omega }\nabla \cdot \left[ x_{j}\mathrm{Im}\left( \overline{\psi }%
\nabla \psi \right) \right] +\int_{\Omega }\mathrm{Im}\left( \overline{\psi }%
\partial _{x_{j}}\psi \right) \\
& =\int_{\Omega }\mathrm{Im}\left( \overline{\psi }\partial _{x_{j}}\psi
\right) \text{,}
\end{align*}%
since $\psi (t,x)=0$ on $\partial \Omega $. Thus we have the \textit{%
momentum law} 
\begin{equation}
\overset{.}{q}(t)=\mathrm{Im}\left( \int_{\Omega }\overline{\psi }\nabla \psi
dx\right) \text{.}  \label{50}
\end{equation}

For the second derivative of the center of mass, we have, by (\ref{T}) and
by divergence theorem, 
\begin{align*}
\overset{..}{q}_{j}(t)=& \frac{d}{dt}\int_{\Omega }\mathrm{Im}\left( \overline{%
\psi }\partial _{x_{j}}\psi dx\right) =-\int_{\Omega }\partial
_{x_{k}}T_{jk}(t,x)dx-\int_{\Omega }|\psi (t,x)|^{2}\partial _{x_{j}}V(x)dx
\\
=& \int_{\partial \Omega }T_{jk}(t,x)\cdot n_{k}d\sigma -\int_{\Omega }|\psi
(t,x)|^{2}\partial _{x_{j}}V(x)dx:=I_{1}+I_{2}\text{,}
\end{align*}%
$n$ being the inward normal to $\partial \Omega $.

Let us use the polar form $\psi (t,x)=u(t,x)e^{is(t,x)}$. Then 
\begin{equation*}
I_{2}=-\int_{\Omega }u^{2}\partial _{x_{j}}V(x)dx\text{.}
\end{equation*}%
This appears to be a force term depending on the potential $V$. This, when
the concentration parameter $h\rightarrow 0$, gives us the Newtonian law for
the motion of a particle (see \cite{Ben-Ghi-Mich}, where this approach is
used in the whole space $\mathbb{R}^{N}$).

Since $u=0$ on the boundary (and since $W(0)=0$), the expression of $T$ is
simplified and the term $I_{1}$ becomes

\begin{align*}
I_{1}=& \int_{\partial \Omega }T_{jk}n_{k}d\sigma =\int_{\partial \Omega
}\left( \partial _{x_{j}}u\  \partial _{x_{k}}u-\frac{1}{4}\delta _{jk}\Delta
\left( u^{2}\right) \right) n_{k}d\sigma \\
=& \int_{\partial \Omega }\left( \partial _{x_{j}}u\  \partial _{x_{k}}u-%
\frac{1}{2}\delta _{jk}|\nabla u|^{2}\right) n_{k}d\sigma \text{.}
\end{align*}%
Again, because $u=0$ on the boundary, by implicit function theorem, we have
that $\nabla u$ is orthogonal to $\partial \Omega $. In addition, we have by
defintion $u=|\psi |\geq 0$, so whenever $\nabla u\neq 0$, the inward
pointing normal vector can be written as $n={\frac{\nabla u}{|\nabla u|}}$.
Thus 
\begin{align*}
I_{1}=& \int_{\partial \Omega }\left( \partial _{x_{j}}u\  \partial _{x_{k}}u-%
\frac{1}{2}\delta _{jk}|\nabla u|^{2}\right) \frac{\partial _{x_{k}}u}{%
|\nabla u|}d\sigma \\
=& \int_{\partial \Omega }(\partial _{x_{j}}u|\nabla u|-\frac{1}{2}\partial
_{x_{j}}u|\nabla u|)d\sigma =\frac{1}{2}\int_{\partial \Omega }\partial
_{x_{j}}u|\nabla u|d\sigma = \\
=& \frac{1}{2}\int_{\partial \Omega }|\nabla u|^{2}n_{j}d\sigma .
\end{align*}

Concluding, we have 
\begin{equation}
\overset{..}{q}(t)=-\int_{\Omega }u^{2}\nabla V(x)dx+\frac{1}{2}%
\int_{\partial \Omega }|\nabla u|^{2}nd\sigma \text{.}  \label{59}
\end{equation}%
In the case of a bounded domain with Dirichlet boundary condition, it
appears an extra term, which represents the centripetal force.
Unfortunately, there are some obvious computational challenges concerning
the last integral of (\ref{59}), and we cannot give a simple expression of
this term, when $h\rightarrow 0$. As it was said in the Introduction, these
challenges call for further work on the dynamics of the solution of (\ref{h}%
).

\textbf{Acknowledgement }M. Magiropoulos wishes to thank G. M. Kavoulakis
and J. Smyrnakis for many helpful discussions.

\end{document}